\documentclass{article}

\usepackage{hyperref}

\usepackage{fullpage}
\usepackage[authoryear,sort]{natbib}
\usepackage[utf8]{inputenc} \usepackage[T1]{fontenc}  \usepackage{etoolbox}
\newtoggle{icml}
\togglefalse{icml}
\usepackage{microtype}    \usepackage{comment}
\usepackage{abstract}     \usepackage{amsmath,amssymb,mathtools,amsfonts}
\usepackage{amsthm}
\usepackage{nicefrac}       \usepackage{graphicx,subfig,color,wrapfig,epsfig,epstopdf}
\usepackage{enumitem}
\usepackage{booktabs}

\usepackage{algorithm,algorithmicx,algpseudocode}

\usepackage{hyperref}       \usepackage{cleveref}
\usepackage{url}            \usepackage{booktabs}       \usepackage{listings}
\usepackage{multirow}
\usepackage{stackengine}
\newcommand\numberthis{\addtocounter{equation}{1}\tag{\theequation}}

\allowdisplaybreaks[4]

\newcounter{ass_counter}
\newcounter{thm_counter}
\newcounter{remark_counter}
\newtheorem{theorem}[thm_counter]{Theorem}\newtheorem{lemma}[thm_counter]{Lemma}\newtheorem{corollary}[thm_counter]{Corollary}
\newtheorem{assumption}[ass_counter]{Assumption}
\newtheorem{remark}[remark_counter]{Remark}
\Crefname{assumption}{Assumption}{Assumptions}
\algblockdefx[BlockUntil]{BlockUntil}{EndBlockUntil}[1][]{\textbf{wait until}
  #1 \textbf{then}}{\textbf{end wait until}} \usepackage[math]{anttor}
\usepackage{antpolt}
\usepackage{caption}

\title{Asynchronous Decentralized Parallel Stochastic Gradient Descent}

\date{\today}

\newcommand\CoAuthorMark{\footnotemark[\arabic{footnote}]}

\usepackage{authblk}
 \author[1]{Xiangru Lian\footnote{Contributed equally.}}
 \author[2]{Wei Zhang\protect\CoAuthorMark}
 \author[3]{Ce Zhang}
 \author[1, 4]{Ji Liu}
 \affil[ ]{\texttt{xiangru@yandex.com, weiz@us.ibm.com, ce.zhang@inf.ethz.ch, ji.liu.uwisc@gmail.com}}
 \affil[1]{Department of Computer Science, University of Rochester}
 \affil[2]{IBM T. J. Watson Research Center}
 \affil[3]{Department of Computer Science, ETH Zurich}
 \affil[4]{Tencent AI Lab}

\begin{document}

\maketitle

\begin{abstract}

Most commonly used distributed machine learning systems are either synchronous
or centralized asynchronous. Synchronous algorithms like AllReduce-SGD perform
poorly in a heterogeneous environment, while asynchronous algorithms using a
parameter server suffer from 1) communication bottleneck at parameter servers
when workers are many, and 2) significantly worse convergence when the traffic
to parameter server is congested. {\em Can we design an algorithm that is robust
  in a heterogeneous environment, while being communication efficient and
  maintaining the best-possible convergence rate?} In this paper, we propose an
asynchronous decentralized stochastic gradient decent algorithm (AD-PSGD)
satisfying all above expectations. Our theoretical analysis shows AD-PSGD
converges at the optimal $O(1/\sqrt{K})$ rate as SGD and has linear speedup
w.r.t. number of workers. Empirically, AD-PSGD outperforms the best of
decentralized parallel SGD (D-PSGD), asynchronous parallel SGD (A-PSGD), and
standard data parallel SGD (AllReduce-SGD), often by orders of magnitude in a
heterogeneous environment. When training ResNet-50 on ImageNet with up to 128
GPUs, AD-PSGD converges (w.r.t epochs) similarly to the
AllReduce-SGD, but each epoch can be up to 4-8$\times$ faster than its synchronous counterparts
in a network-sharing HPC environment. \iftoggle{icml}{}{To the best of our
  knowledge, AD-PSGD is the first asynchronous algorithm that achieves a similar
  epoch-wise convergence rate as AllReduce-SGD, at an over 100-GPU scale.}

 \end{abstract}

\section{Introduction}
\label{sec:introduction}

It often takes hours to train large deep learning tasks such as ImageNet, even
with hundreds of GPUs~\cite{facebook-1hr}. At this scale, how workers
communicate becomes a crucial design choice. Most existing systems such as
TensorFlow \citep{abadi2016tensorflow}, MXNet \citep{chen2015mxnet}, and CNTK
\citep{cntk} support two communication modes: (1) synchronous communication via
parameter servers or AllReduce, or (2) asynchronous communication via parameter
servers. When there are stragglers (i.e., slower workers) in the system, which
is common especially at the scale of hundreds devices, asynchronous approaches
are more robust. However, most asynchronous implementations have a {\em
  centralized} design, as illustrated in Figure~\ref{fig:decen-graphs}(a) --- a
central server holds the shared model for all other workers. Each worker
calculates its own gradients and updates the shared model asynchronously. The
parameter server may become a communication bottleneck and slow down the
convergence. We focus on the question: {\em Can we remove the central server
  bottleneck in asynchronous distributed learning systems while maintaining the
  best possible convergence rate?}

\begin{figure}[t]
  \centering
  \iftoggle{icml}{
  \includegraphics[width=0.45\textwidth]{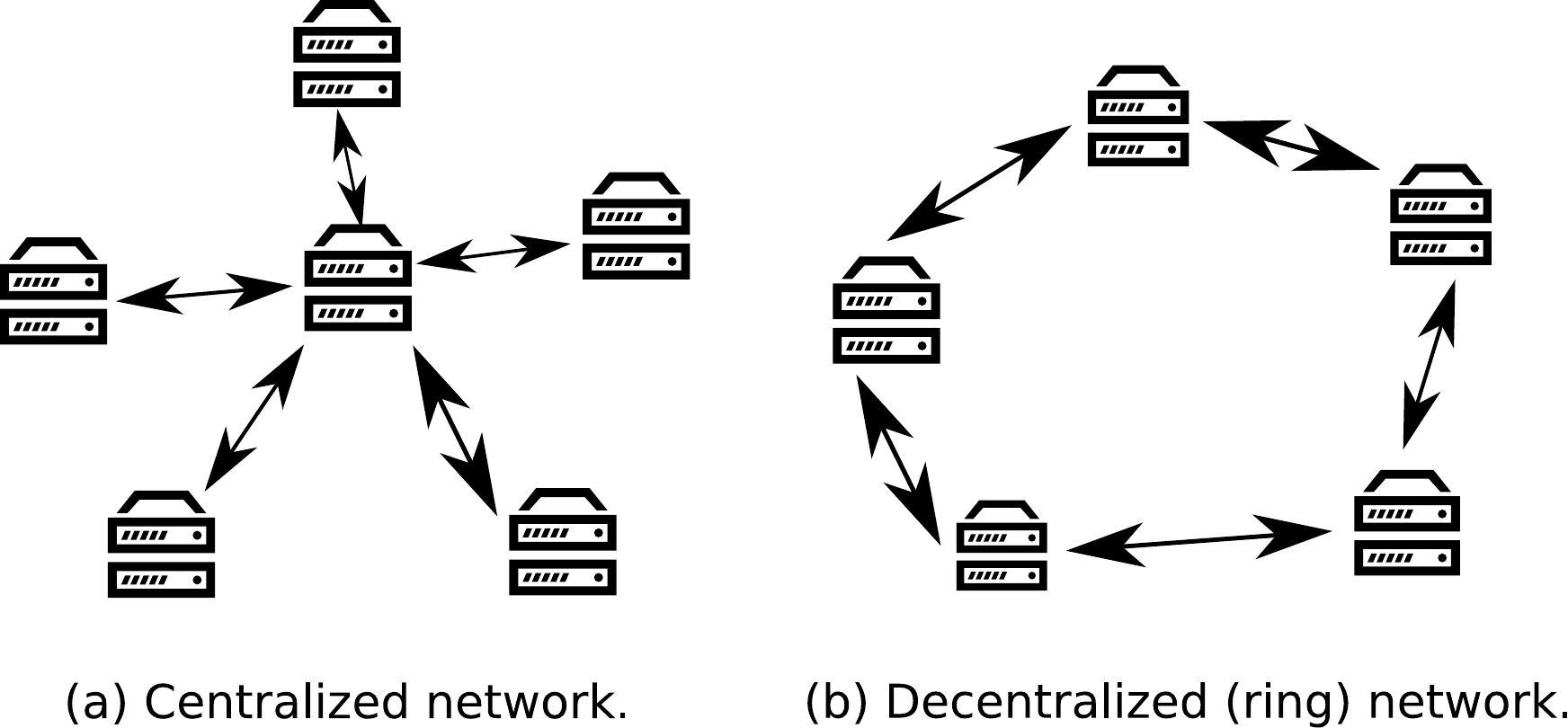}
}{
  \includegraphics[width=0.6\textwidth]{figures/decentr-networks.pdf}
}
  \caption{Centralized network and decentralized network.}
  \label{fig:decen-graphs}
\end{figure}

\iftoggle{icml}{\begin{table}
    \centering
    \begin{minipage}{1.0\linewidth}
      \resizebox{\linewidth}{!}{
        \begin{tabular}{c | l | l}
          \  & \stackanchor{Communication complexity}{(n.t./n.h.)}\footnote{n.t. means number of gradients/models transferred at the busiest worker per $n$ (minibatches of) stochastic gradients updated. n.h. means number of handshakes at the busiest worker per $n$ (minibatches of) stochastic gradients updated.} & Idle time\\
          \hline
          S-PSGD \citep{ghadimi2016mini} & Long ($O(n)$/$O(n)$) & Long\\
          A-PSGD \citep{Lian:2015:APS:2969442.2969545} & Long ($O(n)$/$O(n)$) & Short\\
          AllReduce-SGD \citep{luehr_2017} & Medium ($O(1)/O(n)$) & Long\\
          D-PSGD \citep{1705.09056} & Short ($O(\deg(G))/O(\deg(G))$) & Long\\
          \hline
          \textbf{AD-PSGD} (this paper) & Short ($O(\deg(G))/O(\deg(G))$) & Short
        \end{tabular}
      }
    \end{minipage}
  \caption{\label{tab:comp-algos}Comparison of different distributed machine
      learning algorithms on a network graph $G$. \emph{Long idle time} means in each iteration the
      whole system needs to wait for the slowest worker. \emph{Short idle time}
      means the corresponding algorithm breaks this synchronization per
      iteration. Note that if $G$ is a ring network as required in
      AllReduce-SGD, $O(\deg(G))= O(1)$.}
    \label{tab:comp}
  \end{table}
}{\begin{figure*}[t]
\noindent \begin{minipage}[t]{1.0\textwidth}
  \begin{table}[H]
    \centering
     \begin{tabular}{c | l | l}
       \  & Communication complexity (n.t./n.h.)\footnote{n.t. means number of gradients/models transferred at the busiest worker per $n$ (minibatches of) stochastic gradients updated. n.h. means number of handshakes at the busiest worker per $n$ (minibatches of) stochastic gradients updated.} & Idle time\\
       \hline
       S-PSGD \citep{ghadimi2016mini} & Long ($O(n)$/$O(n)$) & Long\\
       A-PSGD \citep{Lian:2015:APS:2969442.2969545} & Long ($O(n)$/$O(n)$) & Short\\
       AllReduce-SGD \citep{luehr_2017} & Medium ($O(1)/O(n)$) & Long\\
       D-PSGD \citep{1705.09056} & Short ($O(\deg(G))/O(\deg(G))$) & Long\\
       \hline
       \textbf{AD-PSGD} (this paper) & Short ($O(\deg(G))/O(\deg(G))$) & Short
    \end{tabular}
    \caption{\label{tab:comp-algos}Comparison of different distributed machine
      learning algorithms on a network graph $G$. \emph{Long idle time} means in each iteration the
      whole system needs to wait for the slowest worker. \emph{Short idle time}
      means the corresponding algorithm breaks this synchronization per
      iteration. Note that if $G$ is a ring network as required in
      AllReduce-SGD, $O(\deg(G))= O(1)$.}
    \label{tab:comp}
  \end{table}
\end{minipage}
\end{figure*}
}

Recent work \citep{1705.09056} shows that {\em synchronous decentralized}
parallel stochastic gradient descent (D-PSGD) can achieve comparable convergence
rate as its centralized counterparts without any central bottleneck.
\Cref{fig:decen-graphs}-(b) illustrates one communication topology of D-PSGD in
which each worker only talks to its neighbors. However, the synchronous nature
of D-PSGD makes it vulnerable to stragglers because of the synchronization
barrier at each iteration among {\em all} workers. {\em Is it possible to get
  the best of both worlds of asynchronous SGD and decentralized SGD?}

In this paper, we propose the \emph{asynchronous decentralized} parallel
stochastic gradient decent algorithm (AD-PSGD) that is theoretically justified
to keep the advantages of both asynchronous SGD and decentralized SGD. In
AD-PSGD, workers do not wait for all others and only communicate in a
decentralized fashion. AD-PSGD can achieve linear speedup with respect to the
number of workers and admit a convergence rate of $O(1/\sqrt{K})$, where $K$ is
the number of updates. This rate is consistent with D-PSGD and centralized
parallel SGD. By design, AD-PSGD enables wait-free computation and
communication, which ensures \textit{AD-PSGD always converges better (w.r.t
  epochs or wall time) than D-PSGD as the former allows much more frequent
  information exchanging.}

In practice, we found that AD-PSGD is particularly useful in heterogeneous
computing environments such as cloud-computing, where computing/communication
devices' speed often varies. We implement AD-PSGD in Torch and MPI and evaluate
it on an IBM S822LC cluster of up to 128 P100 GPUs. We show that, on real-world
datasets such as ImageNet, AD-PSGD has the same empirical convergence rate as
its centralized and/or synchronous counterpart. In heterogeneous environments,
AD-PSGD can be faster than its fastest synchronous counterparts by orders of
magnitude. On an HPC cluster with homogeneous computing devices but shared
network, AD-PSGD can still outperform its synchronous counterparts by 4X-8X.

Both the theoretical analysis and system implementations of AD-PSGD are
non-trivial, and they form the two technical contributions of this work.

\section{Related work}
\label{sec:related-work}

We review related work in this section. In the following, $K$ and $n$ refer to
the number of iterations and the number of workers, respectively. A comparison of
the algorithms can be found in \Cref{tab:comp-algos}.

The \emph{Stochastic Gradient Descent} (\textbf{SGD})
\citet{nemirovski2009robust, moulines2011non, ghadimi2013stochastic} is a
powerful approach to solve large scale machine learning problems, with the
optimal convergence rate $O(1/\sqrt{K})$ on nonconvex problems.

For \emph{Synchronous Parallel Stochastic Gradient Descent} (\textbf{S-PSGD}),
every worker fetches the model saved in a parameter server and computes a
minibatch of stochastic gradients. Then they push the stochastic gradients to
the parameter server. The parameter server synchronizes all the stochastic
gradients and update their average into the model saved in the parameter server,
which completes one iteration. The convergence rate is proved to be
$O(1/\sqrt{nK})$ on nonconvex problems \citep{ghadimi2016mini}. Results on
convex objectives can be found in \citet{dekel2012optimal}. \iftoggle{icml}{}{Due to the
synchronization step, all other workers have to stay idle to wait for the slowest
one. In each iteration the parameter server needs to synchronize $O(n)$ workers,
which causes high communication cost at the parameter server especially when $n$
is large.}

The \emph{Asynchronous Parallel Stochastic Gradient Descent} (\textbf{A-PSGD})
\citep{recht2011hogwild,agarwal2011distributed,feyzmahdavian2016asynchronous,paine2013gpu}
breaks the synchronization in S-PSGD by allowing workers to use stale weights to
compute gradients. \iftoggle{icml}{}{Asynchronous algorithms significantly reduce the communication overhead by
avoiding idling any worker and can still work well when part of the computing
workers are down.}
On nonconvex problems, when the staleness of the weights used is
upper bounded, A-PSGD is proved to admit the same convergence rate as S-PSGD
\citep{Lian:2015:APS:2969442.2969545,lianNIPS2016_6551}. 

In \emph{AllReduce Stochastic Gradient Descent implementation}
(\textbf{AllReduce-SGD}) \citep{luehr_2017,patarasuk2009bandwidth,mpi_2017}, the
update rule per iteration is exactly the same as in S-PSGD, so they share the
same convergence rate. \iftoggle{icml}{However, there is no parameter server and all the workers use AllReduce to
  synchronize the stochastic gradients.}{However, there is no parameter server in AllReduce-SGD. The workers are connected
  with a ring network and each worker keeps the same local copy of the model. In
  each iteration, each worker calculates a minibatch of stochastic gradients. Then
  all the workers use AllReduce to synchronize the stochastic gradients, after
  which each worker will get the average of all stochastic gradients.}
 In this procedure, only $O(1)$ amount of gradient is
sent/received per worker, but $O(n)$ handshakes are needed on each worker. This
makes AllReduce slow on high latency network.\iftoggle{icml}{}{ At the end of the iteration the
averaged gradient is updated into the local model of each worker.}
 Since we still
have synchronization in each iteration, the idle time is still high as in
S-PSGD.

In \emph{Decentralized Parallel Stochastic Gradient Descent} (\textbf{D-PSGD})
\citep{1705.09056}, all workers are connected with a network that forms a
connected graph $G$. Every worker has its local copy of the model. In each
iteration, all workers compute stochastic gradients locally and at the same time
average its local model with its neighbors. Finally the locally computed
stochastic gradients are updated into the local models. In this procedure, the
busiest worker only sends/receives $O(\deg(G))$ models and has $O(\deg(G))$
handshakes per iteration. \iftoggle{icml}{}{Note that in D-PSGD the computation and communication
can be done in parallel, which means, when communication time is smaller than
the computation time, the communication can be completely hidden.}
The idle time is still high in D-PSGD because all workers need to finish updating
before stepping into the next iteration. Before \citet{1705.09056} there are also
previous studies on decentralized stochastic algorithms (both synchronous and
asynchronous versions) though \emph{none of them is proved to have speedup when
  the number of workers increases}. For example, \citet{lan2017communication}
proposed a decentralized stochastic primal-dual type algorithm with a
computational complexity of $O(n/\epsilon^{2})$ for general convex objectives
and $O(n/\epsilon)$ for strongly convex objectives. \citet{sirb2016consensus}
proposed an asynchronous decentralized stochastic algorithm with a
$O(n/\epsilon^{2})$ complexity for convex objectives. These bounds do not imply
any speedup for decentralized algorithms. {\citet{bianchi2013performance}}
proposed a similar decentralized stochastic algorithm. The authors provided a
convergence rate for the consensus of the local models when the local models are
bounded. \iftoggle{icml}{}{The convergence to a solution was provided by using central limit
theorem. }
However, they did not provide the convergence rate to the solution. A very
recent paper \citep{tang2018d} extended D-PSGD so that it works better on data
with high variance.
\citet{ram2010asynchronous} proposed an asynchronous subgradient variations of
the decentralized stochastic optimization algorithm for convex problems. The
asynchrony was modeled by viewing the update event as a Poisson process and the
convergence to the solution was shown.
\citet{srivastava2011distributed,sundhar2010distributed} are similar. The main
differences from this work are 1) we take the situation where a worker
calculates gradients based on old model into consideration, which is the case in
the asynchronous setting; 2) we prove that our algorithm can achieve linear
speedup when we increase the number of workers, which is important if we want to
use the algorithm to accelerate training; 3) Our implementation guarantees
deadlock-free, wait-free computation and communication. \citet{wildfire}
proposed another distributed stochastic algorithm, but it requires a centralized
arbitrator to decide which two workers are exchanging weights and it lacks
convergence analysis. \citet{tsianos2016efficient} proposed a gossip based dual
averaging algorithm that achieves linear speedup in the computational
complexity, but in each iteration it requires multiple rounds of communication
to limit the difference between all workers within a small constant.

We next briefly review \emph{decentralized algorithms}. Decentralized algorithms
were initially studied by the control community for solving the consensus
problem where the goal is to compute the mean of all the data distributed on
multiple nodes
{\citep{boyd2005gossip,carli2010gossip,aysal2009broadcast,fagnani2008randomized,olfati2007consensus,schenato2007distributed}}.
For decentralized algorithms used for optimization problems, \citet{lu2010gossip}
proposed two non-gradient-based algorithms for solving one-dimensional
unconstrained convex optimization problems\iftoggle{icml}{.}{ where the objective on each node is
strictly convex, by calculating the inverse function of the derivative of the
local objectives and transmitting the gradients or local objectives to
neighbors, and the algorithms can be used over networks with time-varying
topologies. A convergence rate was not shown but the authors did prove the
algorithms will converge to the solution eventually.}
 {\cite{mokhtari2016dsa}}
proposed a fast decentralized variance reduced algorithm for strongly convex
optimization problems.\iftoggle{icml}{}{ The algorithm is proved to have linear convergence rate
and a nice stochastic saddle point method interpretation is given. However, the
speedup property is unclear and a table of stochastic gradients need to be
stored.}
 {\cite{yuan2016convergence}} studied decentralized gradient descent on
convex and strongly convex objectives.\iftoggle{icml}{}{ The algorithm in each iteration averages
the models of the nodes with their neighbors' and then updates the full gradient
of the local objective function on each node.}
 The subgradient version was
considered in {\citet{nedic2009distributed,ram2009distributed}}. The algorithm is
intuitive and easy to understand. However, the limitation of the algorithm is
that it does not converge to the exact solution because the exact solution is
not a fixed point of the algorithm's update rule. This issue was fixed later by
{\citet{shi2015extra,wu2016decentralized}} by using the gradients of last two
instead of one iterates in each iteration, which was later improved in
\citet{shi2015proximal,li2017decentralized} by considering proximal gradients.
Decentralized ADMM algorithms were analyzed in
\citet{zhang2014asynchronous,shi2014linear,aybat2015distributed}.
\citet{wang2016decentralized} develops a decentralized algorithm for recursive
least-squares problems.

\section{Algorithm}

We introduce the AD-PSGD algorithm in this section.

\paragraph{Definitions and notations} Throughout this paper, we use the
following notation and definitions:
\begin{itemize}
\item $\|\cdot\|$ denotes the vector $\ell_{2}$ norm or the matrix spectral norm
  depending on the argument.
\item $\|\cdot\|_F$ denotes the matrix Frobenius norm.
\item $\nabla f(\cdot)$ denotes the gradient of a function $f$.
\item ${\bf 1}_{n}$ denotes the column vector in $\mathbb{R}^{n}$ with $1$ for all elements.
\item $f^{*}$ denotes the optimal solution to \eqref{eq:4zljdslaf}.
\item $\lambda_{i}(\cdot)$ denotes the $i$-th largest eigenvalue of a matrix.
 \item $e_{i}$ denotes the $i$th element of the standard basis of $\mathbb{R}^{n}$.
\end{itemize}

\subsection{Problem definition}

The decentralized
communication topology is represented as an undirected graph: $(V,E)$, where
$V:=\{1,2,\ldots, n\}$ denotes the set of $n$ workers and $E\subseteq V\times V$
is the set of the edges in the graph. Each worker represents a machine/gpu owning its
local data (or a sensor collecting local data online) such that each worker is
associated with a local loss function
\[
f_{i}(x) := \mathbb{E}_{\xi\sim \mathcal{D}_{i}} F_i(x;\xi),
\]
where $\mathcal{D}_i$ is a distribution associated with the local data at worker
$i$ and $\xi$ is a data point sampled via $\mathcal{D}_i$. The edge means that
the connected two workers can exchange information. For the AD-PSGD algorithm, the
overall optimization problem it solves is
\begin{equation}
\min_{x\in \mathbb{R}^{N}} f(x):= \mathbb{E}_{i\sim \mathcal{I}} f_{i}(x) =
  \sum_{i=1}^{n} p_{i}f_{i}(x),\label{eq:4zljdslaf}
\end{equation}
where $p_{i}$'s define a distribution, that is, $p_i\geq 0$ and $\sum_{i} p_i
=1$, and $p_i$ indicates the updating frequency of worker $i$ or the percentage of
the updates performed by worker $i$. The faster a worker, the higher the
corresponding $p_i$. The intuition is that if a worker is faster than another
worker, then the faster worker will run more epochs given the same amount of time,
and consequently the corresponding worker has a larger impact.

\begin{remark}\small
  To solve the common form of objectives in machine learning using AD-PSGD
  \[
    \min_{x\in \mathbb{R}^N}~\mathbb{E}_{\xi\sim \mathcal{D}} F(x; \xi),
  \]
  we can appropriately distribute data such that Eq.~\eqref{eq:4zljdslaf} solves
  the target objective above:
  \begin{description}
  \item[Strategy-1] Let $\mathcal{D}_i=\mathcal{D}$ and $\mathcal{D}$, that is,
    all worker can access all data, and consequently $F_i(\cdot;\cdot) = F(\cdot;
    \cdot)$, that is, all $f_i(\cdot)$'s are the same;
  \item[Strategy-2] Split the data into all workers appropriately such that the
    portion of data is $p_i$ on worker $i$ and define $\mathcal{D}_i$ to be the
    uniform distribution over the assigned data samples.
  \end{description}
\end{remark}

\subsection{AD-PSGD algorithm}

The AD-PSGD algorithm can be described in the following: each worker maintains a
local model $x$ in its local memory and (using worker $i$ as an example) repeats
the following steps:
\begin{itemize}
\item \textbf{Sample data:} Sample a mini-batch of training data denoted by
  $\{\xi_{m}^{i}\}_{m=1}^{M}$, where $M$ is the batch size.
\item \textbf{Compute gradients:} Use the sampled data to compute the stochastic
  gradient $\sum_{m=1}^{M}\nabla F(\hat{x}^i; \xi_m^{i})$, where $\hat{x}^i$ is read
  from the model in the local memory.
\item \textbf{Gradient update:} Update the model in the local memory by $x^i
  \leftarrow {\color{blue} x^i} - \gamma \sum_{m=1}^{M}\nabla F(\hat{x}^i;
  \xi_m^{i})$. Note that $\hat{x}^i$ may not be the same as ${\color{blue} x^i}$
  as it may be modified by other workers in the \textbf{averaging} step.
\item \textbf{Averaging:} Randomly select a neighbor (e.g. worker ${i'}$) and average
  the local model with the worker $i'$'s model $x^{i'}$ (both models on
  both workers are updated to the averaged model). More specifically, $x^i, x^{i'}
  \leftarrow \frac{x^i}{2} + \frac{x^{i'}}{2}$.
\end{itemize}

Note that each worker runs the procedure above on its own without any global
synchronization. This reduces the idle time of each worker and the training
process will still be fast even if part of the network or workers slow down.

The \textbf{averaging} step can be generalized into the following update for all
workers:
\[
[x^1, x^2, \ldots, x^n] \leftarrow [x^1, x^2, \ldots, x^n] W
\]
where $W$ can be an arbitrary doubly stochastic matrix. This generalization
gives plenty flexibility to us in implementation without hurting our analysis.

All workers run the procedure above simultaneously, as shown in
\Cref{alg:Async-D-PSGD}. We use a virtual counter $k$ to denote the iteration
counter -- every single \textbf{gradient update} happens no matter on which worker
will increase $k$ by $1$. $i_k$ denotes the worker performing the $k$th update.

\subsection{Implementation details}

We briefly describe two interesting aspects of system designs
and leave more discussions to \Cref{sec:add:impl}.

\subsubsection{Deadlock avoidance}

A naive implementation of the above algorithm may cause deadlock --- the
averaging step needs to be atomic and involves updating two workers (the
selected worker and one of its neighbors). As an example, given three fully
connected workers $A$, $B$, and $C$, $A$ sends its local model $x_{A}$ to $B$
and waits for $x_{B}$ from $B$; $B$ has already sent out $x_{B}$ to $C$ and
waits for $C$'s response; and $C$ has sent out $x_C$ to $A$ and waits for $x_A$
from $A$.

We prevent the deadlock in the following way: The
communication network is designed to be a bipartite graph,
that is, the worker set $V$ can be split into two disjoint sets $A$ (active set)
and $P$ (passive set) such that any edge in the graph connects one worker in $A$
and one worker in $P$.
Due to the property of the bipartite graph, the neighbors of any active worker
can only be passive workers and the neighbors of any passive worker can only be
active workers. This implementation avoids deadlock but still fits in the
general algorithm \Cref{alg:Async-D-PSGD} we are analyzing. We leave more
discussions and a detailed implementation for wait-free training to
\Cref{sec:add:impl}.

\subsubsection{Communication topology}

The simplest realization of AD-PSGD algroithm is a ring-based topology. To
accelerate information exchanging, we also implement a communication topology
in which each sender communicates with a reciever that is $2^{i}+1$ hops away in the
ring, where $i$ is an integer from 0 to $\log(n-1)$ ($n$ is the number of
learners). It is easy to see it takes at most $O(\log(n))$
steps for any pair of workers to exchange information instead of $O(n)$ in the
simple ring-based topology. In this way, $\rho$ (as defined in
\Cref{sec:theory}) becomes smaller and the scalability of AD-PSGD improves. This
implementation also enables robustness against slow or failed network links
because there are multiple routes for a worker to disseminate its information.

\begin{algorithm}
  \caption{AD-PSGD (logical view) \label{alg:Async-D-PSGD}}
  \begin{minipage}{\iftoggle{icml}{0.48\textwidth}{\textwidth}}
    \iftoggle{icml}{
      \small
    }{}
    \begin{algorithmic}[1]
      \Require Initialize local models $\{x_{0}^{i}\}_{i=1}^{n}$ with the same initialization, learning rate
      $\gamma$      , batch size $M$, and total number of iterations $K$.
      \For {$k=0,1,\ldots, K-1$}
        \State Randomly sample a worker $i_{k}$ of the graph $G$ and randomly
        sample an averaging matrix $W_{k}$ which can be dependent on $i_{k}$.
        \State Randomly sample a batch
        \iftoggle{icml}{\[\xi_{k}^{i_{k}}:=(\xi_{k,1}^{i_{k}}, \xi_{k,2}^{i_{k}},
            \ldots, \xi_{k,M}^{i_{k}})\]}{$\xi_{k,i_{k}}:=(\xi_{k,1}^{i_{k}},
          \xi_{k,2}^{i_{k}}, \ldots, \xi_{k,M}^{i_{k}})$} from local data of the
        $i_{k}$-th worker.
        \State \label{step:C} Compute the stochastic gradient locally
        \iftoggle{icml}{\[g_k(\hat{x}_{k}^{i_{k}};\xi_{k}^{i_{k}}) :=
            \sum_{j=1}^{M}\nabla F(\hat{x}_{k}^{i_{k}};
            \xi_{k,j}^{i_{k}})\]}{$g_k(\hat{x}_{k}^{i_{k}};\xi_{k}^{i_{k}}) :=
          \sum_{j=1}^{M}\nabla F(\hat{x}_{k}^{i_{k}}; \xi_{k,j}^{i_{k}})$}.
        \State \label{step:A}
        Average local models by \footnote{Note that \Cref{step:C}
          and \Cref{step:A} can run in parallel.} \iftoggle{icml}{\[[x^1_{k+1/2}, x^2_{k+1/2},
            \ldots, x^n_{k+1/2}] \leftarrow [x^1_k, x^2_k, \ldots, x^n_k]
            W_k\]}{$[x^1_{k+1/2}, x^2_{k+1/2}, \ldots, x^n_{k+1/2}] \leftarrow
          [x^1_k, x^2_k, \ldots, x^n_k] W_k$}
        \State Update the local model \iftoggle{icml}{\[x_{k+1}^{i_{k}} \gets x_{k+1/2}^{i_{k}} - \gamma
            g_{k}(\hat{x}_{k}^{i_{k}};\xi_{k}^{i_{k}}),\] \[x_{k+1}^{j}\gets
            x_{k+1/2}^{j},\forall j\neq i_{k}.\]}{$x_{k+1}^{i_{k}} \gets x_{k+1/2}^{i_{k}} - \gamma
          g_{k}(\hat{x}_{k}^{i_{k}};\xi_{k}^{i_{k}}) $ and $x_{k+1}^{j}\gets
          x_{k+1/2}^{j}, \forall j\neq i_{k}.$}
      \EndFor
      \State Output the average of the models on all workers for inference.
    \end{algorithmic}
  \end{minipage}
\end{algorithm}

\section{Theoretical analysis}
\label{sec:theory}
In this section we provide theoretical analysis for the AD-PSGD
algorithm. We will show that the convergence rate of AD-PSGD is consistent with SGD and D-PSGD.

Note that by counting each update of stochastic gradients as one iteration, the
update of each iteration in \Cref{alg:Async-D-PSGD} can be viewed as
\[X_{k+1} = X_k W_k -\gamma \partial g(\hat{X}_{k}; \xi_{k}^{i_{k}}, i_k),\]
where $k$ is the iteration number, $x_{k}^{i}$ is the local model of the $i$th
worker at the $k$th iteration, and
\begin{align*}
  X_k = & \left[\begin{array}{ccc} x_{k}^{1} & \cdots & x_{k}^{n} \end{array}\right] \in \mathbb{R}^{N\times n},\\
  \hat{X}_k = & \left[\begin{array}{ccc} \hat{x}_{k}^{1} & \cdots & \hat{x}_{k}^{n} \end{array}\right] \in \mathbb{R}^{N\times n},\\
\iftoggle{icml}{  \partial g (\hat{X}_{k} ; \xi_{k}^{i_{k}}, i_k) = & {\scriptstyle \big[
    0 \enspace \cdots \enspace 0 \enspace \sum_{j=1}^{M}\nabla F (\hat{x}_{k}^{i_{k}}, \xi_{k,j}^{ i_k}) \enspace 0 \enspace \cdots \enspace 0 \big]} \in \mathbb{R}^{N\times n},
}{  \partial g (\hat{X}_{k} ; \xi_{k}^{i_{k}}, i_k) = & \left[\begin{array}{ccccccc}
    0 & \cdots & 0 & \sum_{j=1}^{M}\nabla F (\hat{x}_{k}^{i_{k}}, \xi_{k,j}^{ i_k}) & 0 & \cdots & 0
  \end{array}\right]\in \mathbb{R}^{N\times n},
}
\end{align*}
and $\hat{X}_k= X_{k-\tau_{k}}$ for some nonnegative integer $\tau_{k}$.

\begin{assumption}
\label{ass:20170614-002451}
Throughout this paper, we make the following commonly used assumptions:
  \begin{enumerate}
    \item \textbf{Lipschitzian gradient:} All functions $f_{i}(\cdot)$'s are
      with $L$-Lipschitzian gradients.
    \item \textbf{Doubly stochastic averaging:} $W_k$ is doubly stochastic for all $k$.
    \item \textbf{Spectral gap:} There exists a $\rho \in [0,1)$ such
      that\footnote{A smaller $\rho$ means a faster information spreading in
        the network, leading to a faster convergence.}
      \begin{equation}
      \max\{|\lambda_{2}(\mathbb{E}[W_{k}^\top W_k])|,
      |\lambda_{n}(\mathbb{E}[W_{k}^\top W_k])|\}\le \rho,\forall k.\label{eq:2lkzjsdf}
    \end{equation}
    \item \textbf{Unbiased estimation:}
\footnote{Note that this is easily satisfied when all workers can access all data so
      that $\mathbb{E}_{\xi \sim \mathcal{D}_{i}}\nabla F(x;\xi) = \nabla
      f(x)$. When each worker can only access part of the data, we can also meet
      these assumptions by appropriately distributing data.}
      \begin{align}
        \mathbb{E}_{\xi \sim \mathcal{D}_{i}}\nabla F(x;\xi) &= \nabla f_{i}(x),\label{eq:azjlzj}\\
        \mathbb{E}_{i\sim \mathcal{I}}\mathbb{E}_{\xi \sim \mathcal{D}_{i}}\nabla F(x;\xi) &= \nabla f(x).
      \end{align}
  \item \textbf{Bounded variance:} Assume the variance of the stochastic
      gradient
      \[\mathbb{E}_{i\sim \mathcal{I}}\mathbb{E}_{\xi \sim
          \mathcal{D}_{i}}\|\nabla F(x;\xi) - \nabla f(x)\|^{2}\] is bounded for
      any $x$ with $i$ sampled from the distribution $\mathcal{I}$ and $\xi$
      from the distribution $\mathcal{D}_{i}$. This implies there exist
      constants $\sigma$ and $\varsigma$ such that
      \begin{align}
        \mathbb{E}_{\xi\sim \mathcal{D}_{i}} \| \nabla F(x, \xi) - \nabla f_i (x) \|^2 & \leqslant \sigma^2, \forall i,\forall x .\label{eq:20170613-221110}\\
        \mathbb{E}_{i\sim \mathcal{I}}\|\nabla f_{i}(x) - \nabla f(x)\|^{2}&\leqslant \varsigma^{2},\forall x.\label{eq:1}
      \end{align}
    Note that if all workers can access all data, then $\varsigma = 0$.
     \item \textbf{Dependence of random variables:} $\xi_{k}, i_{k}, k\in \{0, 1, 2, \ldots\}$
       are independent random variables. $W_{k}$ is a random variable dependent
       on $i_{k}$.
     \item \textbf{Bounded staleness:} $\hat{X}_{k}= X_{k-\tau_{k}}$ and there
       exists a constant $T$ such that $\max_k \tau_{k}\leqslant T$.   \end{enumerate}
\end{assumption}

Throughout this paper, we define the following notations for simpler notation
\iftoggle{icml}{
  \begin{small}
    \begin{align*}
      \bar{\rho} := & \frac{n-1}{n}\left(  \frac{1}{1 - \rho} + \frac{2 \sqrt{\rho}}{\left( 1 - \sqrt{\rho} \right)^2}  \right),\\
      C_1 := & 1 - 24 M^{2} L^2 \gamma^2 \left( T \frac{n - 1}{n} + \bar{\rho} \right),\\
      C_2 := & \frac{\gamma M}{2 n} - \frac{\gamma^2 LM^2}{n^2} - \frac{2 M^3 L^2
               T^2 \gamma^3}{n^3} \\ & \hspace{-0.7cm}-  \frac{4\left( \frac{6 \gamma^2 L^3 M^2}{n^2} + \frac{\gamma
               M}{n} L^2 + \frac{12 M^3 L^4 T^2 \gamma^3}{n^3} \right)  M^2
               \gamma^2 (T\frac{n-1}{n}+ \bar{\rho})}{C_1},\\
      C_3 := & \frac{1}{2} + \frac{2 \left( 6 \gamma^2 L^2 M^2 + \gamma nML +
               \frac{12 M^3 L^3 T^2 \gamma^3}{n} \right)  \bar{\rho}}{C_1}  +
               \frac{LT^2 \gamma M}{n} .
    \end{align*}
  \end{small}

}{\begin{align*}
  \bar{\rho} := & \frac{n-1}{n}\left(  \frac{1}{1 - \rho} + \frac{2 \sqrt{\rho}}{\left( 1 - \sqrt{\rho} \right)^2}  \right),\\
  C_1 := & 1 - 24 M^{2} L^2 \gamma^2 \left( T \frac{n - 1}{n} + \bar{\rho} \right),\\
  C_2 := & \frac{\gamma M}{2 n} - \frac{\gamma^2 LM^2}{n^2} - \frac{2 M^3 L^2
            T^2 \gamma^3}{n^3} - \left( \frac{6 \gamma^2 L^3 M^2}{n^2} + \frac{\gamma
            M}{n} L^2 + \frac{12 M^3 L^4 T^2 \gamma^3}{n^3} \right)  \frac{4 M^2
            \gamma^2 (T\frac{n-1}{n}+ \bar{\rho})}{C_1},\\
C_3 := & \frac{1}{2} + {2 {C_1^{-1}} \left( 6 \gamma^2 L^2 M^2 + \gamma nML +
           \frac{12 M^3 L^3 T^2 \gamma^3}{n} \right)  \bar{\rho}} +
           \frac{LT^2 \gamma M}{n} .
\end{align*}
}

Under \Cref{ass:20170614-002451} we have the following results:

\begin{theorem}[Main theorem]\label{thm:main-thm}
  While $C_{3}\leqslant 1$ and $C_{2}\geqslant 0$ and $C_{1}> 0$ are satisfied
  we have
\iftoggle{icml}{\begin{equation*}
  \resizebox{\columnwidth}{!}{$
    \frac{\sum_{k = 0}^{K - 1} \mathbb{E} \left\| \nabla f \left( \frac{X_k
    \mathbf{1}_n}{n} \right) \right\|^2}{K} \leqslant
   \frac{2 (\mathbb{E} f (x_0) - \mathbb{E} f^{\ast}) n}{\gamma KM} + \frac{2 \gamma L(\sigma^2 + 6 M \varsigma^2)}{n}  .
    $}
\end{equation*}
}{
\begin{align*}
  \frac{\sum_{k = 0}^{K - 1} \mathbb{E} \left\| \nabla f \left( \frac{X_k
    \mathbf{1}_n}{n} \right) \right\|^2}{K} \leqslant
    & \frac{2 (\mathbb{E} f (x_0) - \mathbb{E} f^{\ast}) n}{\gamma KM} + \frac{2 \gamma L}{n}  (
      \sigma^2 + 6 M \varsigma^2).
\end{align*}
}
\end{theorem}  Noting that $\frac{X_{n}\mathbf{1}_{n}}{n} =
\frac{1}{n}\sum_{i=1}^{n}x_{k}^{i}$, this theorem characterizes the convergence
of the average of all local models. By appropriately choosing the learning rate,
we obtain the following corollary
\begin{corollary}
  \label{coro:Sun Aug 13 00:14:48 EDT 2017}
  Let $\gamma = \frac{n}{10ML + \sqrt{\sigma^2 + 6 M \varsigma^2} \sqrt{KM}}$.
  We have the following convergence rate

\iftoggle{icml}{{\small \begin{align} &\frac{\sum_{k = 0}^{K - 1} \mathbb{E} \left\| \nabla f \left( \frac{X_k
   \mathbf{1}_n}{n} \right) \right\|^2}{K}\nonumber \\ \hspace{-5mm} \leqslant & \frac{20 ( f
   (x_0) -  f^{\ast}) L}{K} + \frac{2 (f (x_0) -
    f^{\ast} + L)  \sqrt{\sigma^2/M + 6 \varsigma^2}}{\sqrt{K}}\label{eq:2}
\end{align}}
}{\begin{align} \frac{\sum_{k = 0}^{K - 1} \mathbb{E} \left\| \nabla f \left( \frac{X_k
   \mathbf{1}_n}{n} \right) \right\|^2}{K} \leqslant& \frac{20 ( f
   (x_0) -  f^{\ast}) L}{K} + \frac{2 (f (x_0) -
    f^{\ast} + L)  \sqrt{\sigma^2/M + 6 \varsigma^2}}{\sqrt{K}}\label{eq:2}
\end{align}
}
if the total number of iterations is sufficiently large, in particular,
\iftoggle{icml}{{\small \begin{align}
 K &\geqslant \frac{ML^2 n^2}{\sigma^2 + 6 M \varsigma^2}
  \max \Bigg\{192 \left( T \frac{n - 1}{n} + \bar{\rho} \right), \frac{64 T^4}{n^2},\label{eq:3} \\
    & 1024 n^2  \bar{\rho}^2, \frac{\left( 8 \sqrt{6} T^{2 / 3} + 8 \right)^2  \left( T + \bar{\rho}
    \frac{n}{n - 1} \right)^{2 / 3}(n - 1)^{1 / 2}}{n^{1 / 6} } \Bigg\} .\nonumber
\end{align} }
}{\begin{equation}
K \geqslant \frac{ML^2 n^2}{\sigma^2 + 6 M \varsigma^2}
  \max \left\{ \begin{array}{c}
    192 \left( T \frac{n - 1}{n} + \bar{\rho} \right), \frac{64 T^4}{n^2},
    1024 n^2  \bar{\rho}^2, \frac{\left( 8 \sqrt{6} T^{2 / 3} + 8 \right)^2  \left( T + \bar{\rho}
    \frac{n}{n - 1} \right)^{2 / 3}(n - 1)^{1 / 2}}{n^{1 / 6} }
  \end{array} \right\} .\label{eq:3}
\end{equation}
}
\end{corollary}
This corollary indicates that if the iteration number is big enough, AD-PSGD's
convergence rate is $O(1/\sqrt{K})$. We compare the convergence rate of AD-PSGD
with existing results for SGD and D-PSGD to show the tightness of the proved
convergence rate. We will also show the efficiency and the linear speedup
property for AD-PSGD w.r.t. batch size, number of workers, and staleness
respectively. Further discussions on communication topology and intuition will
be provided at the end of this section.
\begin{remark}[Consistency with SGD]
  Note that if $T=0$ and $n=1$ the proposed AD-PSGD reduces to the vanilla SGD
  algorithm \cite{nemirovski2009robust, moulines2011non, ghadimi2013stochastic}.
  Since $n=1$, we do not have the variance among workers, that is, $\varsigma =
  0$, the convergence rate becomes $O(1/K + \sigma / \sqrt{KM})$ which is
  consistent with the convergence rate with SGD.
\end{remark}

\begin{remark}[Linear speedup w.r.t. batch size]
  When $K$ is large enough the second term on the RHS of \eqref{eq:2} dominates
  the first term. Note that the second term converges at a rate $O(1/\sqrt{MK})$
  if $\varsigma = 0$, which means the convergence efficiency gets boosted with a
  linear rate if increase the mini-batch size. This observation indicates the
  linear speedup w.r.t. the batch size and matches the results of mini-batch
  SGD. \footnote{ Note that when $\varsigma^{2} \neq 0$, AD-PSGD does not admit
    this linear speedup w.r.t. batch size. It is unavoidable because increasing
    the minibatch size only decreases the variance of the stochastic gradients
    within each worker, while $\varsigma^{2}$ characterizes the variance of
    stochastic gradient among different workers, independent of the batch size.
  }\end{remark}

\begin{remark}[Linear speedup w.r.t. number of workers]
  Note that every single stochastic gradient update counts one iteration in our
  analysis and our convergence rate in \Cref{coro:Sun Aug 13 00:14:48 EDT 2017}
  is consistent with SGD / mini-batch SGD. It means that the number of required
  stochastic gradient updates to achieve a certain precision is consistent with
  SGD / mini-batch SGD, as long as the total number of iterations is large
  enough. It further indicates the linear speedup with respect to the number of
  workers $n$ ($n$ workers will make the iteration number advance $n$ times
  faster in the sense of wall-clock time, which means we will converge $n$ times
  faster). To the best of our knowledge, the linear speedup property w.r.t. to
  the number of workers for decentralized algorithms has not been recognized
  until the recent analysis for D-PSGD by \citet{1705.09056}. Our analysis
  reveals that by breaking the synchronization AD-PSGD can maintain linear
  speedup, reduce the idle time, and improve the robustness in heterogeneous
  computing environments.
\end{remark}

\begin{remark}[Linear speedup w.r.t. the staleness]
  From \eqref{eq:3} we can also see that as long as the staleness $T$ is bounded
  by $O(K^{1/4})$ (if other parameters are considered to be constants), linear
  speedup is achievable.
\end{remark}

\section{Experiments}
\label{sec:exp}

We describe our experimental methodologies in \Cref{sec:exp:meth} and we evaluate
the AD-PSGD algorithm in the following sections:
\begin{itemize}
\item \Cref{sec:exp:conv}: Compare AD-PSGD's convergence rate (w.r.t epochs) with other algorithms.
\item \Cref{sec:exp:speedup}: Compare AD-PSGD's convergence rate (w.r.t runtime) and its speedup with other algorithms.
\item \Cref{sec:exp:robustness}: Compare AD-PSGD's robustness to other algorithms in heterogeneous
  computing and heterogeneous communication environments.
\item \Cref{sec:nlc-experiments}: Evaluate AD-PSGD on IBM proprietary natural
  language processing dataset and model.
\end{itemize}

\subsection{Experiments methodology}
\label{sec:exp:meth}
\subsubsection{Dataset, model, and software}

We use CIFAR10 and ImageNet-1K as the evaluation dataset and we use Torch-7 as our deep
learning framework. We use MPI to implement the communication scheme. For CIFAR10, we evaluate both VGG~\citep{vgg} and ResNet-20~\citep{he2016deep}
models. VGG, whose size is about 60MB, represents a communication intensive workload and
ResNet-20, whose size is about 1MB, represents a computation intensive workload.
For the ImageNet-1K dataset, we use the ResNet-50 model whose size is about 100MB.

Additionally, we experimented on an IBM proprietary natural language processing
datasets and models \cite{gadei} in \Cref{sec:nlc-experiments}.

\subsubsection{Hardware}

We evaluate AD-PSGD in two different environments:

\begin{itemize}
\item IBM S822LC HPC cluster: Each node with 4 Nvidia P100 GPUs, 160 Power8
  cores (8-way SMT) and 500GB memory on each node. 100Gbit/s Mellanox EDR
  infiniband network. We use 32 such nodes.
\item x86-based cluster: This cluster is a cloud-like environment with 10Gbit/s
  ethernet connection. Each node has 4 Nvidia P100 GPUs, 56 Xeon E5-2680 cores
  (2-way SMT), and 1TB DRAM. We use 4 such nodes.
\end{itemize}

\subsubsection{Compared algorithms}

We compare the proposed AD-PSGD algorithm to AllReduce-SGD, D-PSGD
\cite{1705.09056} and a state of the art asynchronous SGD implementation EAMSGD.
\cite{eamsgd}\footnote{In this paper, we use ASGD and EAMSGD interchangeably.}
In EAMSGD, each worker can communicate with the parameter server less frequently
by increasing the ``communication period'' parameter $su$.

\subsection{Convergence w.r.t. epochs}
\label{sec:exp:conv}

\begin{table}
\centering
\caption{Testing accuracy comparison for VGG and ResNet-20 model on CIFAR10. 16
  workers in
  total. }
\label{tab:final-accuracy}
\iftoggle{icml}{\resizebox{0.8\linewidth}{!}{\begin{tabular}{|l|l|l|l|l|l|}
\hline
& AllReduce & D-PSGD & EAMSGD &AD-PSGD \\
  \hline
VGG & 87.04\% & 86.48\% & 85.75\%& 88.58\% \\
  \hline
ResNet-20 & 90.72\% & 90.81\% & 89.82\% & 91.49\% \\
  \hline
\end{tabular}}
}{\begin{tabular}{|l|l|l|l|l|l|}
\hline
& AllReduce & D-PSGD & EAMSGD &AD-PSGD \\
  \hline
VGG & 87.04\% & 86.48\% & 85.75\%& 88.58\% \\
  \hline
ResNet-20 & 90.72\% & 90.81\% & 89.82\% &91.49\% \\
  \hline
\end{tabular}
}
\end{table}
\paragraph{CIFAR10}
\Cref{fig:loss} plots training loss w.r.t. epochs for each algorithm, which is
evaluated for VGG and ResNet-20 models on CIFAR10 dataset with 16 workers.
\Cref{tab:final-accuracy} reports the test accuracy of all algorithms.

For EAMSGD, we did extensive hyper-parameter tuning to get the best possible
model, where $su=1$. We set momentum moving average to be $0.9/n$ (where $n$ is the number of
workers) as recommended in \cite{eamsgd} for EAMSGD.

For other algorithms, we use the following hyper-parameter setup as prescribed
in \cite{cifar-torch} and \cite{fb.resnet}:

\begin{itemize}
\item Batch size: 128 per worker for VGG, 32 for ResNet-20.
\item Learning rate: For VGG start from 1 and reduce by half every 25 epochs.
  For ResNet-20 start from 0.1 and decay by a factor of 10 at the 81st epoch and
  the 122nd epoch.
\item Momentum: 0.9.
\item Weight decay: $10^{-4}$.
\end{itemize}

\emph{\Cref{fig:loss} show that w.r.t epochs, AllReduce-SGD, D-PSGD and AD-PSGD
  converge similar, while ASGD converges worse. \Cref{tab:final-accuracy} shows
  AD-PSGD does not sacrifice test accuracy. }

\iftoggle{icml}{\begin{figure}[t]\small
    \centering
    \subfloat[{\small VGG loss}]{{ \includegraphics[width=0.5\columnwidth]{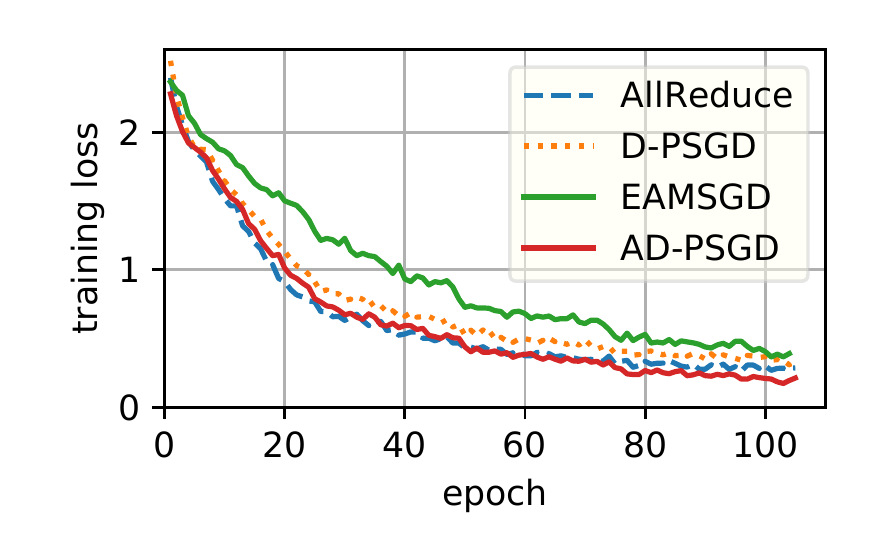} }\label{fig:cifar-vgg-loss}}    \subfloat[{\small ResNet-20 loss}]{{ \includegraphics[width=0.5\columnwidth]{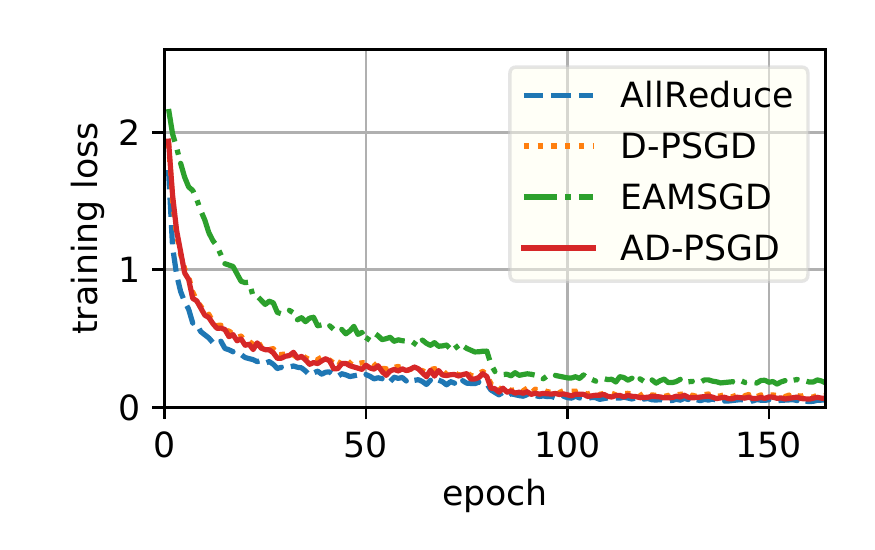} }\label{fig:cifar-resnet20-loss}}    \caption{\small Training loss comparison for VGG and ResNet-20 model on CIFAR10. AllReduce-SGD, D-PSGD and AD-PSGD converge alike, EAMSGD converges the worst. 16 workers in total.}
    \label{fig:loss}
\end{figure}
}{\begin{figure*}[t]\small
    \centering
    \subfloat[{\small VGG loss}]{{ \includegraphics[width=0.3\textwidth]{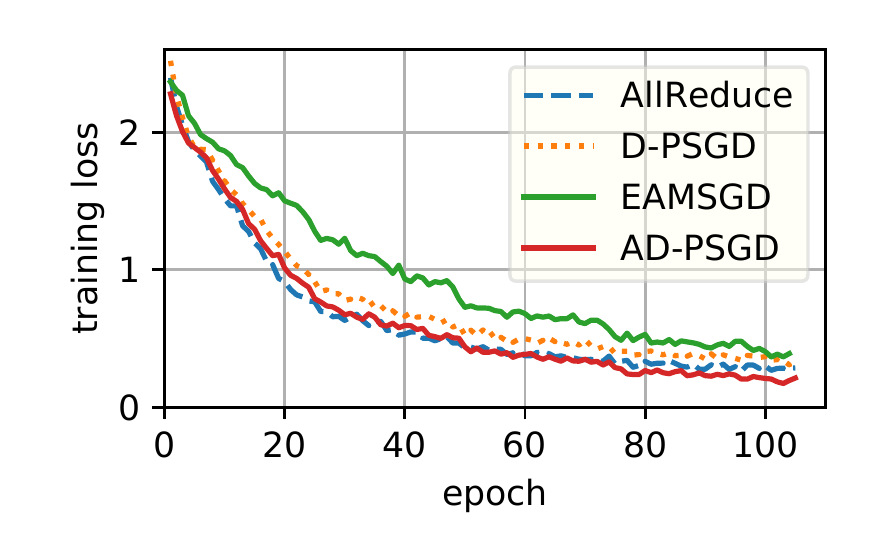} }\label{fig:cifar-vgg-loss}}    \subfloat[{\small ResNet-20 loss}]{{ \includegraphics[width=0.3\textwidth]{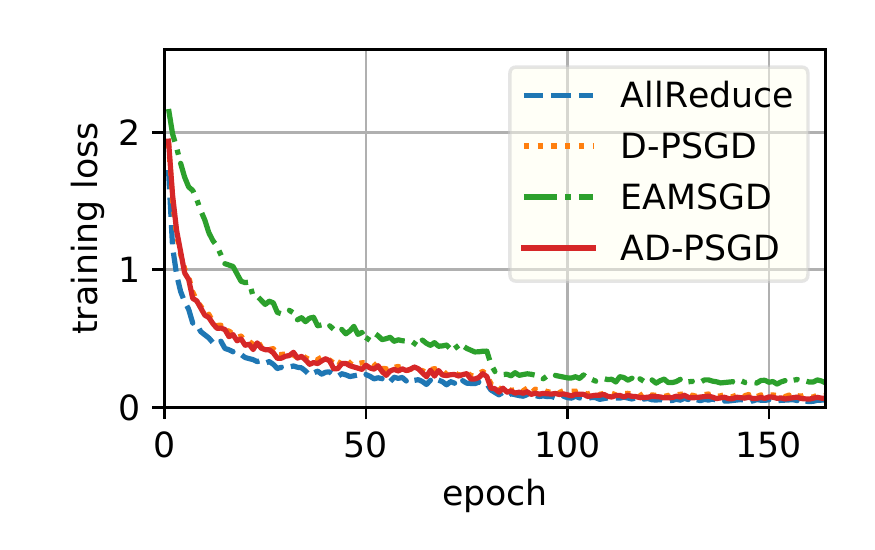} }\label{fig:cifar-resnet20-loss}}    \caption{\small Training loss comparison for VGG and ResNet-20 model on CIFAR10. AllReduce-SGD, D-PSGD and AD-PSGD converge alike, EAMSGD converges the worst. 16 workers in total.}
    \label{fig:loss}
\end{figure*}
}

\iftoggle{icml}{
\begin{table}
\centering
\caption{Testing accuracy comparison for ResNet-50 model on ImageNet dataset for
  AllReduce, D-PSGD, and AD-PSGD. The ResNet-50 model is trained for 90 epochs.
  AD-PSGD and AllReduce-SGD achieve similar model accuracy.}
\label{tab:imagenet-accuracy}
  \resizebox{0.8\columnwidth}{!}{
\begin{tabular}{|l|l|l|l|}
\hline
& AllReduce & D-PSGD & AD-PSGD \\
  \hline
  16 Workers & 74.86\% & 74.74\% & 75.28\% \\
  \hline
  32 Workers & 74.78\% & 73.66\% & 74.66\%\\
  \hline
  64 Workers & 74.90\% & 71.18\% & 74.20\%\\
  \hline
  128 Workers & 74.78\% & 70.90\% & 74.23\% \\
  \hline
\end{tabular}}
\end{table}
}{

\begin{table}[t]
\centering
\caption{Testing accuracy comparison for ResNet-50 model on ImageNet dataset for
  AllReduce, D-PSGD, and AD-PSGD. The ResNet-50 model is trained for 90 epochs.
  AD-PSGD and AllReduce-SGD achieve similar model accuracy.}
\label{tab:imagenet-accuracy}
\begin{tabular}{|l|l|l|l|}
\hline
& AllReduce & D-PSGD & AD-PSGD \\
  \hline
  16 Workers & 74.86\% & 74.74\% & 75.28\% \\
  \hline
  32 Workers & 74.78\% & 73.66\% & 74.66\%\\
  \hline
  64 Workers & 74.90\% & 71.18\% & 74.20\%\\
  \hline
  128 Workers & 74.78\% & 70.90\% & 74.23\% \\
  \hline
\end{tabular}
\end{table}
}

\paragraph{ImageNet} We further evaluate the AD-PSGD's convergence rate w.r.t.
epochs using ImageNet-1K and ResNet-50 model. We compare AD-PSGD with
AllReduce-SGD and D-PSGD as they tend to converge better than A-PSGD.

\Cref{fig:imagenet} and \Cref{tab:imagenet-accuracy} demonstrate that w.r.t.
epochs AD-PSGD converges similarly to AllReduce and converges better than D-PSGD
when running with 16,32,64,128 workers. How to maintain convergence while
increasing $M \times n$\footnote{$M$ is mini-batch size per worker and $n$ is
  the number of workers} is an active ongoing research area
\cite{zhang2016icdm,facebook-1hr} and it is orthogonal to the topic of this
paper. For 64 and 128 workers, we adopted similar learning rate tuning scheme as
proposed in \citet{facebook-1hr} (i.e., learning rate warm-up and linear
scaling)\footnote{In AD-PSGD, we decay the learning rate every 25 epochs instead
  of 30 epochs as in AllReduce.} \emph{It worths noting that we could further
  increase the scalability of AD-PSGD by combining learners on the same
  computing node as a super-learner (via Nvidia NCCL AllReduce collectives). In
  this way, a 128-worker system can easily scale up to 512 GPUs or more,
  depending on the GPU count on a node. }

{\em
  Above results show AD-PSGD converges similarly to AllReduce-SGD w.r.t epochs and
  better than D-PSGD. Techniques used for tuning learning rate for AllReduce-SGD
  can be applied to AD-PSGD when batch size is large.
}

\subsection{Speedup and convergence w.r.t runtime}
\label{sec:exp:speedup}
\iftoggle{icml}{\begin{figure}
\small
    \centering
        { \includegraphics[width=\columnwidth]{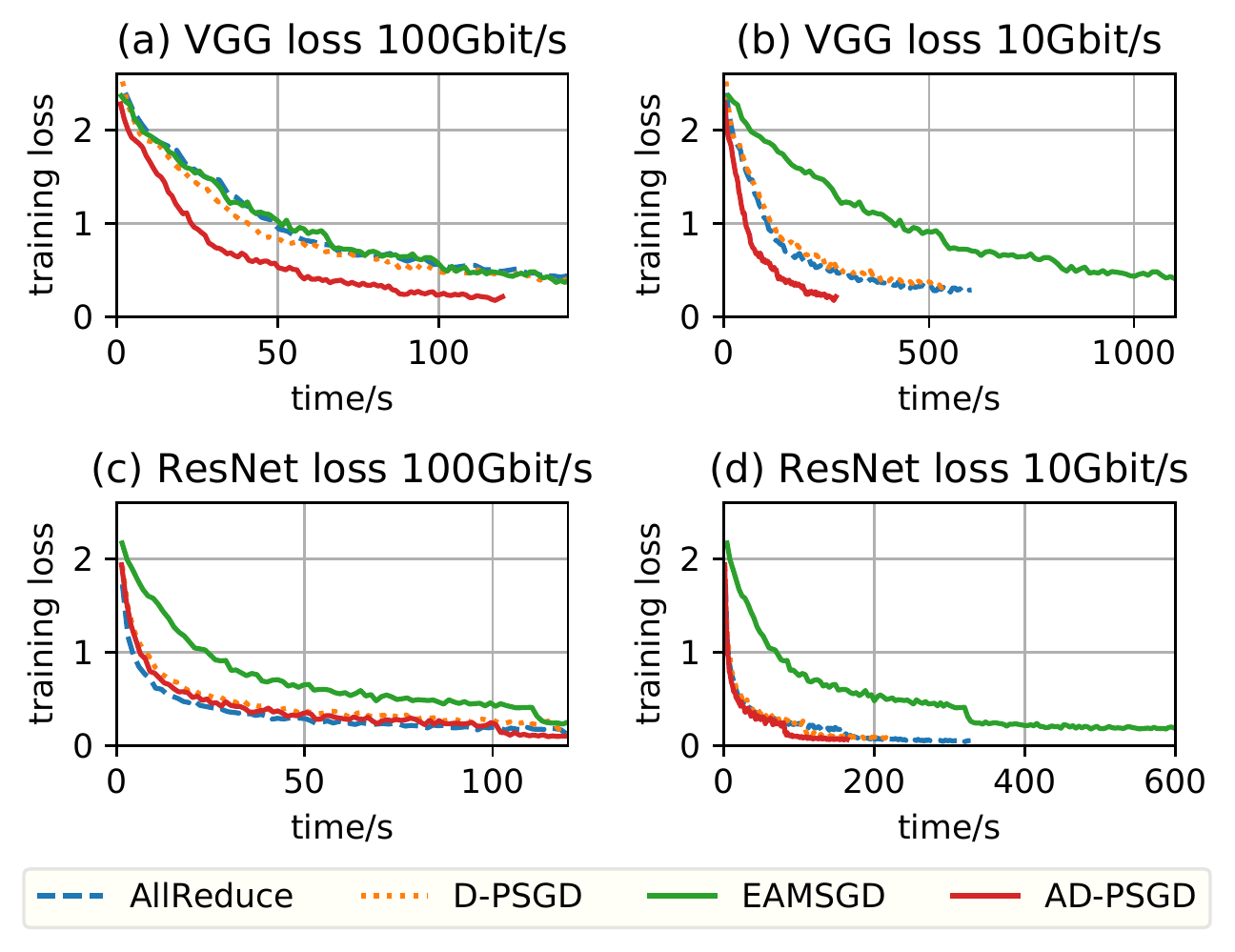}

          \caption{Runtime comparison for VGG (communication intensive) and
            ResNet-20 (computation intensive) models on CIFAR10. Experiments run
            on IBM HPC w/ 100Gbit/s network links and on x86 system w/ 10Gbit/s
            network links. AD-PSGD consistently converges the fastest. 16
            workers in total.}
          \label{fig:runtimes}
        }

\end{figure}
}{\begin{figure*}[t]\small
    \centering
        { \includegraphics[width=\textwidth]{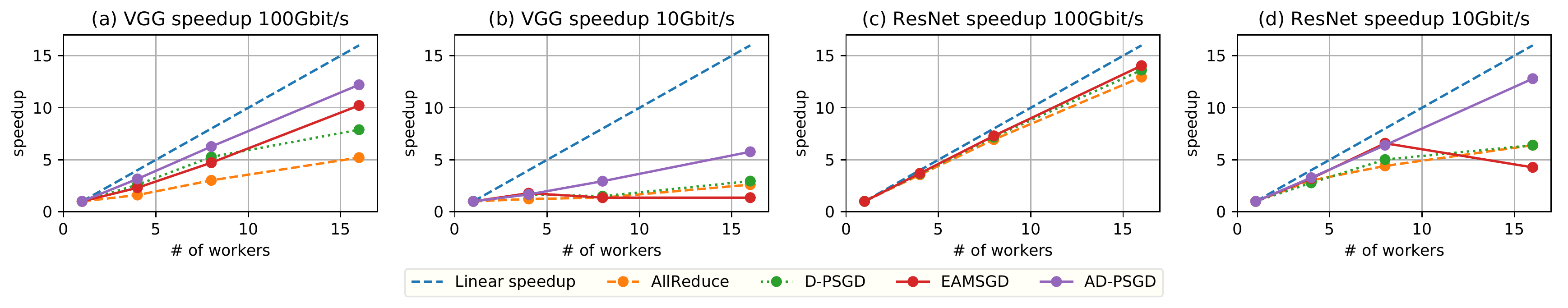}

          \caption{Runtime comparison for VGG (communication intensive) and
            ResNet-20 (computation intensive) models on CIFAR10. Experiments run
            on IBM HPC w/ 100Gbit/s network links and on x86 system w/ 10Gbit/s
            network links. AD-PSGD consistently converges the fastest. 16
            workers in total.}
          \label{fig:runtimes}
        }
\end{figure*}
}

\begin{figure*}[t]
\small
    \centering
    \subfloat[{\small 16 workers }]{{ \includegraphics[width=0.18\textwidth]{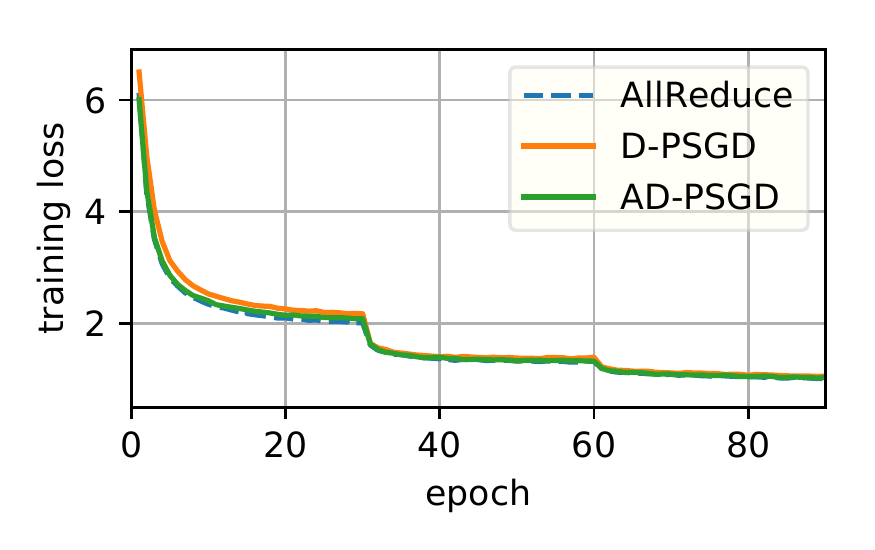} }\label{fig:imagenet-resnet50-16L}}
    \subfloat[{\small 32 workers}]{{ \includegraphics[width=0.18\textwidth]{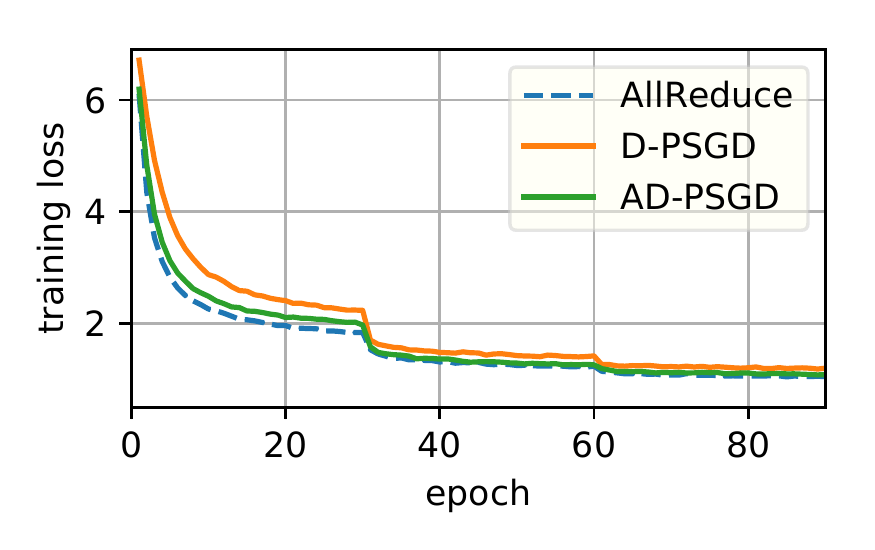} }\label{fig:imagenet-resnet50-32L}}
    \subfloat[{\small 64 workers}]{{ \includegraphics[width=0.18\textwidth]{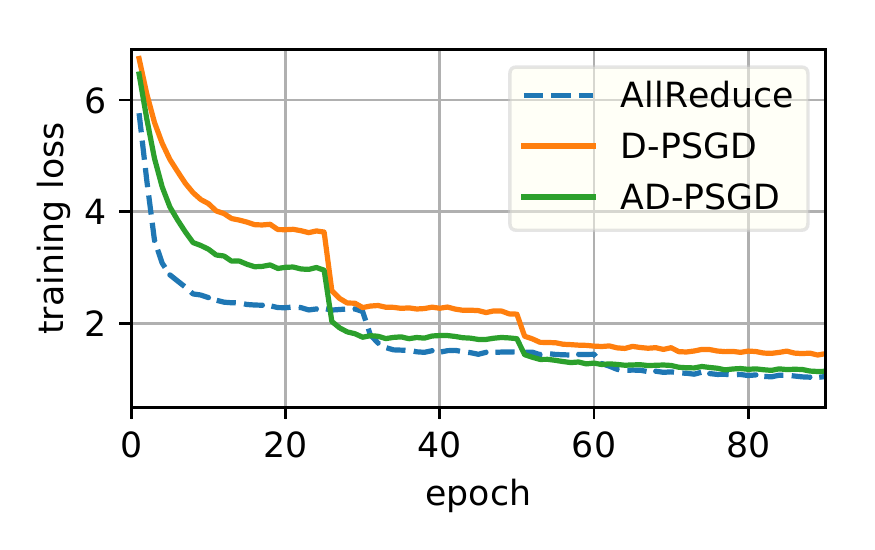} }\label{fig:imagenet-resnet50-64L}}
    \subfloat[{\small 128 workers}]{{ \includegraphics[width=0.18\textwidth]{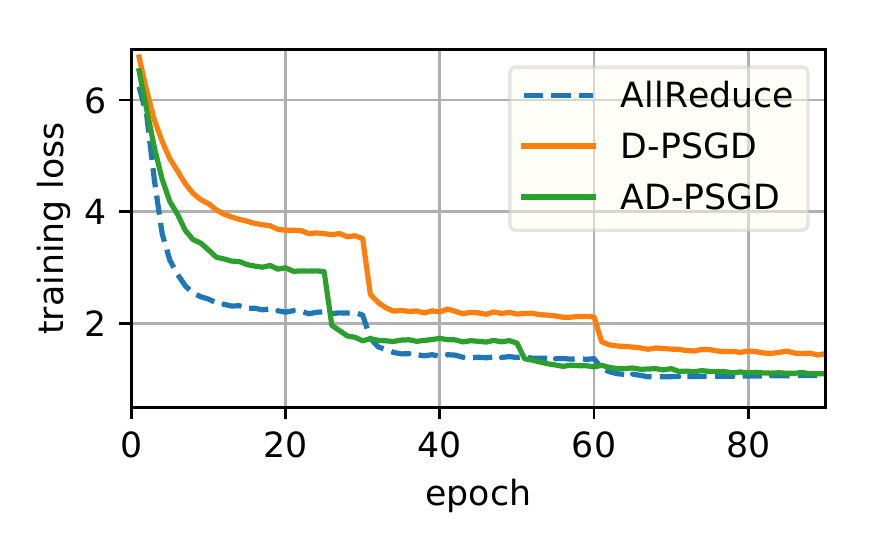} }\label{fig:imagenet-resnet50-1288L}}
    \subfloat[{\small training time variation (64 workers)}]{{ \includegraphics[width=0.18\textwidth]{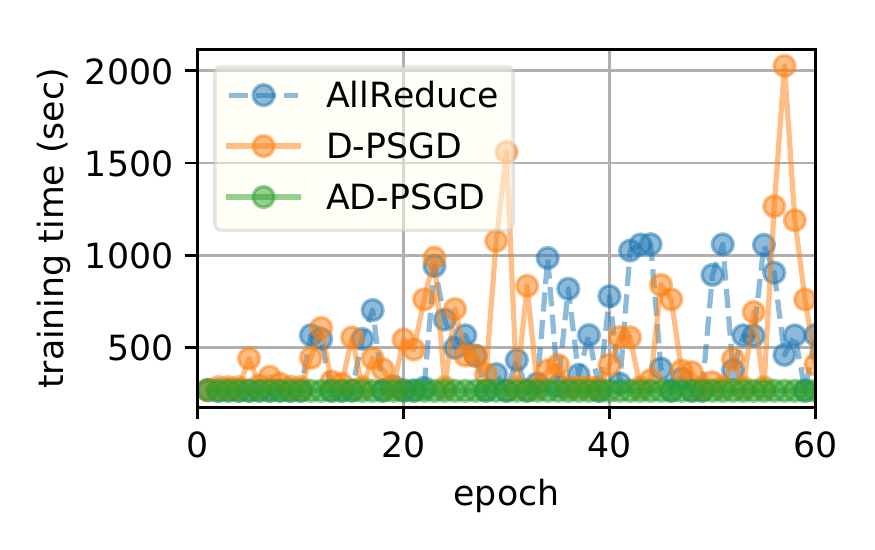} }\label{fig:time-variance-imagenet}}
    \caption{\small Training loss and training time per epoch comparison for
      ResNet-50 model on ImageNet dataset, evaluated up to 128 workers. AD-PSGD
      and AllReduce-SGD converge alike, better than D-PSGD. For 64 workers AD-PSGD finishes
      each epoch in 264 seconds, whereas AllReduce-SGD and D-PSGD can take over
      1000 sec/epoch. }
    \label{fig:imagenet}
\end{figure*}

\iftoggle{icml}{\begin{figure*}[t]
\small
    \centering
        { \includegraphics[width=1.8\columnwidth]{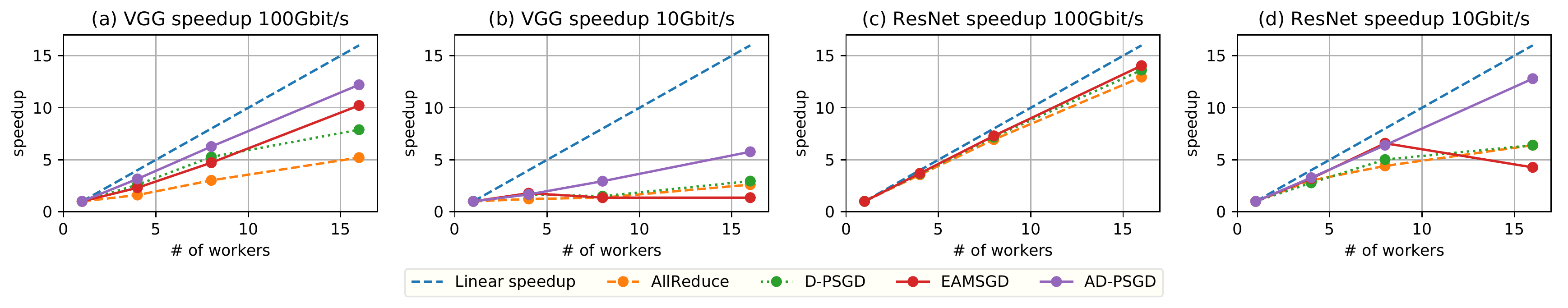}

          \caption{Speedup comparison for VGG (communication intensive) and
            ResNet-20 (computation intensive) models on CIFAR10. Experiments run
            on IBM HPC w/ 100Gbit/s network links and on x86 system w/ 10Gbit/s
            network links. AD-PSGD consistently achieves the best speedup.}
          \label{fig:speedups}
        }\end{figure*}
}{\begin{figure*}[t]\small
    \centering
        { \includegraphics[width=\textwidth]{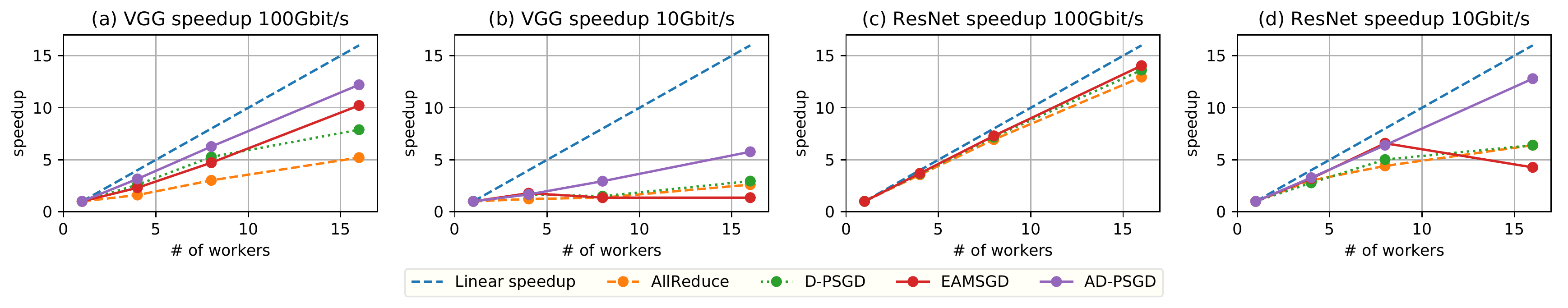}

          \caption{Speedup comparison for VGG (communication intensive) and ResNet-20 (computation intensive) models on CIFAR10. Experiments run on IBM HPC w/ 100Gbit/s network links and on x86 system w/ 10Gbit/s network links. AD-PSGD consistently achieves the best speedup.}
          \label{fig:speedups}        }
\end{figure*}
}

On CIFAR10, \Cref{fig:runtimes} shows the runtime convergence results on both
IBM HPC and x86 system. The EAMSGD implementation deploys parameter server
sharding to mitigate the network bottleneck at the parameter servers. However,
the central parameter server quickly becomes a bottleneck on a slow network with
a large model as shown in \Cref{fig:runtimes}-(b).

\Cref{fig:speedups} shows the speedup for different algorithms w.r.t. number of
workers. The speedup for ResNet-20 is better than VGG because ResNet-20 is a
computation intensive workload.

\emph{Above results show that regardless of workload type (computation intensive
  or communication intensive) and communication networks (fast or slow), AD-PSGD
  consistently converges the fastest w.r.t. runtime and achieves the best
  speedup. }

\subsection{Robustness in a heterogeneous environment}
\label{sec:exp:robustness}
In a heterogeneous environment, the speed of computation device and communication device may often vary, subject
to architectural features (e.g., over/under-clocking, caching, paging),
resource-sharing (e.g., cloud computing) and hardware malfunctions. Synchronous
algorithms like AllReduce-SGD and D-PSGD perform poorly when workers'
computation and/or communication speeds vary. Centralized asynchronous
algorithms, such as A-PSGD, do poorly when the parameter server's network links
slow down. In contrast, AD-PSGD localizes the impact of slower workers or
network
links. 

On ImageNet, \Cref{fig:time-variance-imagenet} shows the epoch-wise training
time of the AD-PSGD, D-PSGD and AllReduce run over 64 GPUs (16 nodes) over a
reserved window of 10 hours when the job shares network links with other jobs on
IBM HPC. AD-PSGD finishes each epoch in 264 seconds, whereas AllReduce-SGD and
D-PSGD can take over 1000 sec/epoch.

We then evaluate AD-PSGD's robustness under different situations by randomly
slowing down 1 of the 16 workers and its incoming/outgoing network links. Due to
space limit, we will discuss the results for ResNet-20 model on CIFAR10 dataset
as the VGG results are similar.

\paragraph{Robustness against slow computation}
\iftoggle{icml}{\begin{figure}[t]\small
    \centering
    \subfloat[A computation device slows down by
    10X-100X.]{{ \includegraphics[width=0.45\columnwidth]{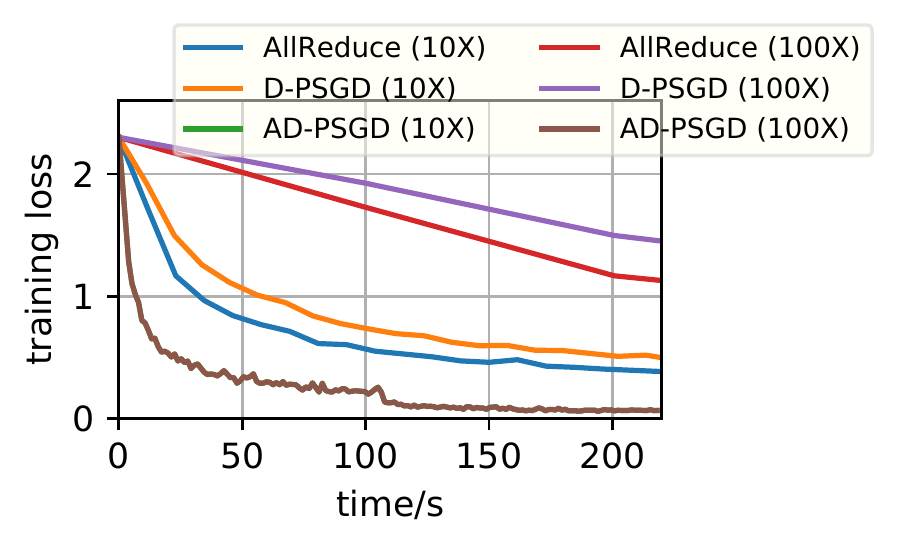}
      }\label{fig:loss-slower-worker}}\quad     \subfloat[A network link slows down by
    2X-100X.]{{ \includegraphics[width=0.45\columnwidth]{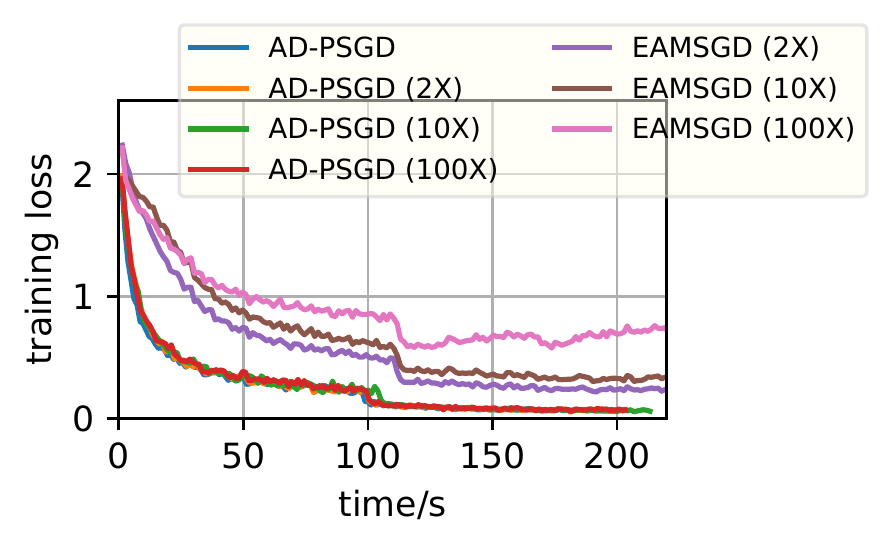}
      }\label{fig:loss-link}}    \\
    \caption{\small Training loss for ResNet-20 model on CIFAR10, when a
      computation device/network link slows down by 2X-100X. $k$X in parentheses means a random worker's GPU or network links slow down by $k$-times. 16 workers in
      total. AD-PSGD is robust under such heterogeneous environments.}
\end{figure}

}{\begin{figure*}[t]\small
    \centering
   \subfloat[AD-PSGD converges steadily w.r.t. epochs despite a slower computing device.]{{ \includegraphics[width=0.3\textwidth]{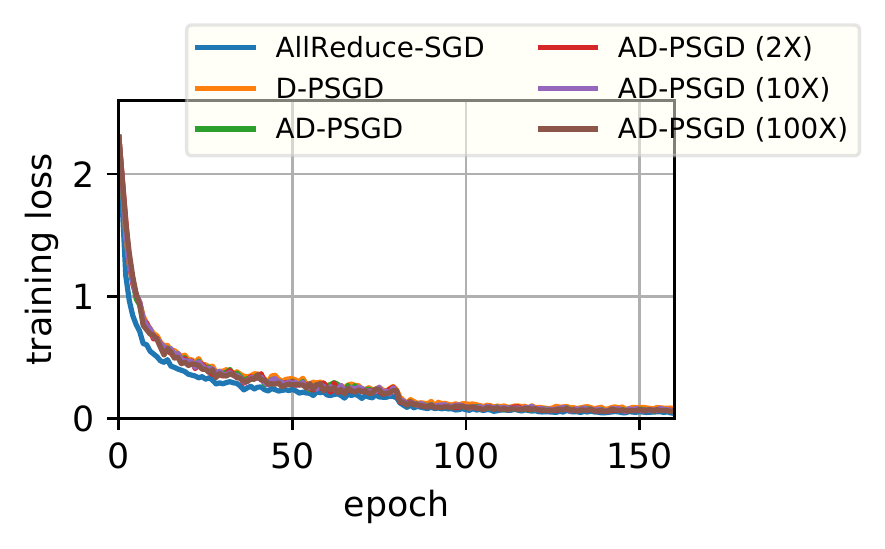} }\label{fig:slower-loss-epoch}} \quad
   \subfloat[AD-PSGD converges much faster than AllReduce-SGD and D-PSGD in the presence of a slower computing device.]{{ \includegraphics[width=0.3\textwidth]{./figures/plot/resnet-slower-runtime} }\label{fig:slower-loss-runtime}} \\
 \caption{\small Training loss for ResNet-20 model on CIFAR10 w.r.t (a) epochs
   and (b) runtime, when a computation device slows down by 2X-100X. 16 workers
   in total.}
    \label{fig:loss-slower-worker}
\end{figure*}
\begin{figure}[t]  \small
  \centering

    {\includegraphics[width=0.4\columnwidth]{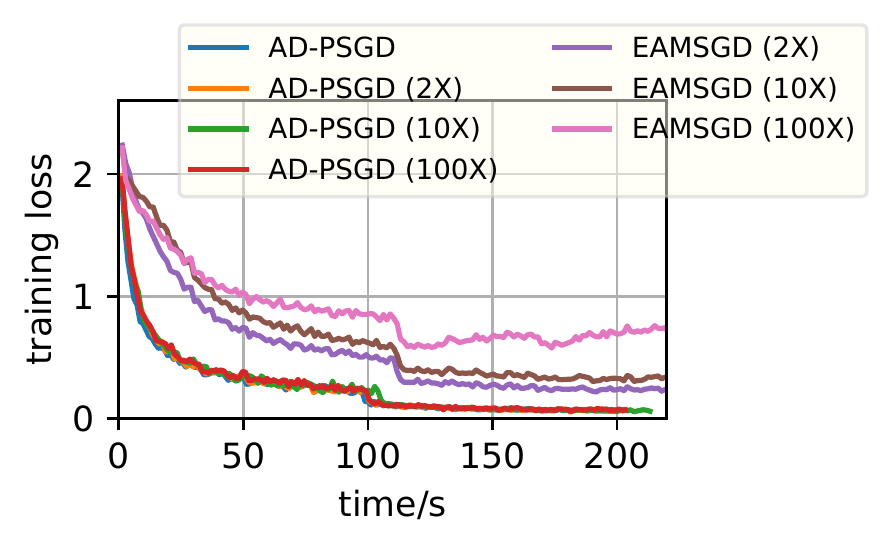}      \caption{\small Training loss for ResNet-20 on CIFAR10 w.r.t. runtime,
        when a network link slows down by 2X-100X.       }
    \label{fig:loss-link}

    }
\end{figure}
}

\iftoggle{icml}{

      \begin{table}
        \centering
        \caption{Runtime comparison for ResNet-20 model on CIFAR10 dataset when
          a worker slows down by 2X-100X.}
        \label{tab:speedup_slow_node}

\resizebox{\linewidth}{!}{
        \begin{tabular}{|l|l|l|l|l|}
          \hline
          \multirow{2}{*}{\begin{tabular}[c]{@{}l@{}}Slowdown of\\ one node\end{tabular}} & \multicolumn{2}{l|}{AD-PSGD} & \multicolumn{2}{l|}{AllReduce/D-PSGD}         \\ \cline{2-5}
                                                                                          & Time/epoch (sec)    & Speedup   & Time/epoch (sec) & Speedup         \\ \hline
          no slowdown                                                                        & 1.22              & 14.78     & 1.47/1.45           & 12.27/12.44           \\ \hline
          2X                                                                                 & 1.28               & 14.09     & 2.6/2.36             & 6.93/7.64            \\ \hline
          10X                                                                                & 1.33               & 13.56     & 11.51/11.24             & 1.56/1.60            \\ \hline
          100X                                                                               & 1.33               & 13.56     & 100.4/100.4             & 0.18/0.18            \\ \hline
        \end{tabular}
        }
      \end{table}
}{\begin{table}[t]
\centering
\caption{Runtime efficiency comparison for ResNet-20 model on CIFAR-10 dataset when a computation device slows down by 2X-100X. AD-PSGD converges faster than AllReduce-SGD and D-PSGD, by orders of magnitude. 16 workers in total.}
\label{tab:speedup_slow_node}
\begin{tabular}{|l|l|l|l|l|}
\hline
\multirow{2}{*}{\begin{tabular}[c]{@{}l@{}}Slowdown of\\ one node\end{tabular}} & \multicolumn{2}{l|}{AD-PSGD} & \multicolumn{2}{l|}{AllReduce/D-PSGD}         \\ \cline{2-5}
                                                                                   & Time/epoch(sec)    & Speedup   & Time/epoch (sec) & Speedup         \\ \hline
no slowdown                                                                        & 1.22              & 14.78     & 1.47/1.45           & 12.27/12.44           \\ \hline
2X                                                                                 & 1.28               & 14.09     & 2.6/2.36             & 6.93/7.64            \\ \hline
10X                                                                                & 1.33               & 13.56     & 11.51/11.24             & 1.56/1.60            \\ \hline
100X                                                                               & 1.33               & 13.56     & 100.4/100.4             & 0.18/0.18            \\ \hline
\end{tabular}
\end{table}
}

\emph{\Cref{fig:loss-slower-worker} and \Cref{tab:speedup_slow_node} shows that
AD-PSGD's convergence is robust against slower workers. AD-PSGD can converge
faster than AllReduce-SGD and D-PSGD by orders of magnitude when there is a very
slow worker. }

\paragraph{Robustness against slow communication}

\Cref{fig:loss-link} shows that AD-PSGD is robust when one worker is connected
to slower network links. In contrast, centralized asynchronous algorithm EAMSGD
uses a larger communication period to overcome slower links, which significantly
slows down the convergence.

\emph{These results show only AD-PSGD is robust against both heterogeneous
  computation and heterogeneous communication. }

\section{Conclusion}

This paper proposes an asynchronous decentralized stochastic gradient descent
algorithm (AD-PSGD). The algorithm is not only robust in heterogeneous
environments by combining both decentralization and asynchronization, but it is
also theoretically justified to have the same convergence rate as its
synchronous and/or centralized counterparts and can achieve linear speedup w.r.t. number of
workers. Extensive experiments validate the proposed algorithm.

{\small
  \paragraph{Acknowledgment}
  This project is supported in part by NSF CCF1718513, NEC fellowship, IBM
  faculty award, Swiss NSF NRP 75 407540\_167266, IBM Zurich, Mercedes-Benz
  Research \& Development North America, Oracle Labs, Swisscom, Zurich
  Insurance, and Chinese Scholarship Council. \iftoggle{icml}{}{We thank David
    Grove, Hillery Hunter, and Ravi Nair for providing valuable feedback. We
    thank Anthony Giordano and Paul Crumley for well-maintaining the computing
    infrastructure that enables the experiments conducted in this paper.} }

\newpage
\bibliographystyle{abbrvnat}

\newpage
\appendix
\clearpage

\section{Wait-free (continuous) training and communication}
\label{sec:add:impl}

The theoretical guarantee of AD-PSGD relies on the the doubly stochastic
property of matrix $W$. The implication is the averaging of the weights between
two workers should be atomic. This brings a special challenge for current
distributed deep learning frameworks where the computation (gradients
calculation and weights update) runs on GPU devices and the communication runs
on CPU (or its peripherals such as infiniband or RDMA), because when there is
averaging happening on a worker, the GPU is not allowed to update gradients into
the weights. This can be solve by using CPU to update weights while GPUs only
calculate gradients. Every worker (including active and passive workers) runs
two threads in parallel with a shared buffer $g$, one thread for computation and
the other for communication. \Cref{alg:Async-D-PSGD-1},
\Cref{alg:Async-D-PSGD-2}, and \Cref{alg:Async-D-PSGD-3} illustrate the task on
each thread. The communication thread is run by CPUs, while the computation
thread is run by GPUs. In this way GPUs can continuously calculate new gradients
by putting the results in CPUs' buffer regardless of whether there is averaging
happening. Recall in D-PSGD, communication only occurs once in each iteration.
In contrast, AD-PSGD can exchange weights at any time by using this
implementation.

\begin{algorithm}
  \caption{Computation thread on active or passive worker (worker index is $i$)\label{alg:Async-D-PSGD-1}}
  \begin{minipage}{\iftoggle{icml}{0.48\textwidth}{\textwidth}}
  \begin{algorithmic}[1]
\iftoggle{icml}{\small }{ }
    \Require Batch size $M$
    \While {not terminated}
      \State Pull model $x^{i}$ from the communication thread.
      \State Update locally in the thread $x^{i} \gets x^{i}-\gamma g$.\footnote{At this time
        the communication thread may have not update $g$ into $x^{i}$ so the
        computation thread pulls an old model. We compensate this by doing local
        update in computation thread. We observe this helps the scaling.}
      \State Randomly sample a batch $\xi^i:=(\xi_{1}^{i}, \xi_{2}^{i},\ldots, \xi_{M}^{i})$ from local data of the $i$-th worker and compute
      the stochastic gradient $g^i(x^i;\xi^i):=\sum_{m=1}^{M}\nabla F(x^{i};\xi^{i}_{m})$ locally.
      \BlockUntil[$g=0$]
      \State Local buffer $g\gets g^{i}(x^{i};\xi^{i})$.\footnote{We can also
        make a queue of gradients here to avoid the waiting. Note that doing
        this will make the effective batch-size different from $M$.}
      \EndBlockUntil
    \EndWhile
  \end{algorithmic}
  \end{minipage}
\end{algorithm}

\begin{algorithm}
  \caption{Communication thread on active worker (worker index is $i$) \label{alg:Async-D-PSGD-2}}
  \begin{minipage}{\iftoggle{icml}{0.48\textwidth}{\textwidth}}
  \begin{algorithmic}[1]
\iftoggle{icml}{\small }{ }
  \Require Initialize local model $x^i$, learning rate $\gamma$.
  \While {not terminated}
    \If {$g\neq 0$}
        \State $x^{i} \gets x^{i} - \gamma g, \quad g\gets 0$.    \EndIf
    \State Randomly select a neighbor (namely worker $j$). Send $x^{i}$ to worker $j$ and fetch $x^{j}$ from it.
    \State $x^{i}\gets \frac{1}{2}(x^{i} + x^{j})$.
  \EndWhile
  \end{algorithmic}
  \end{minipage}
\end{algorithm}

\begin{algorithm}
  \caption{Communication thread on passive worker (worker index is $j$)  \label{alg:Async-D-PSGD-3}}
  \begin{minipage}{\iftoggle{icml}{0.48\textwidth}{\textwidth}}
  \begin{algorithmic}[1]
\iftoggle{icml}{\small }{ }
  \Require Initialize local model $x^j$, learning rate $\gamma$.
  \While {not terminated}
    \If {$g\neq 0$ }
        \State $x^{j} \gets x^{j} - \gamma g, \quad g\gets 0$.     \EndIf
    \If {receive the request of reading local model (say from worker $i$)}
      \State Send $x^{j}$ to worker $i$.
      \State $x^{j} \gets \frac{1}{2}(x^{i}+ x^{j})$.
    \EndIf
  \EndWhile
  \end{algorithmic}
  \end{minipage}
\end{algorithm}

\section{NLC experiments}
\label{sec:nlc-experiments}

In this section, we use IBM proprietary natural language processing dataset and
model to evaluate AD-PSGD against other algorithms.

\iftoggle{icml}{\begin{figure}[H]
\small
    \centering
    \subfloat[{\small NLC-Joule}]{{ \includegraphics[width=0.5\columnwidth]{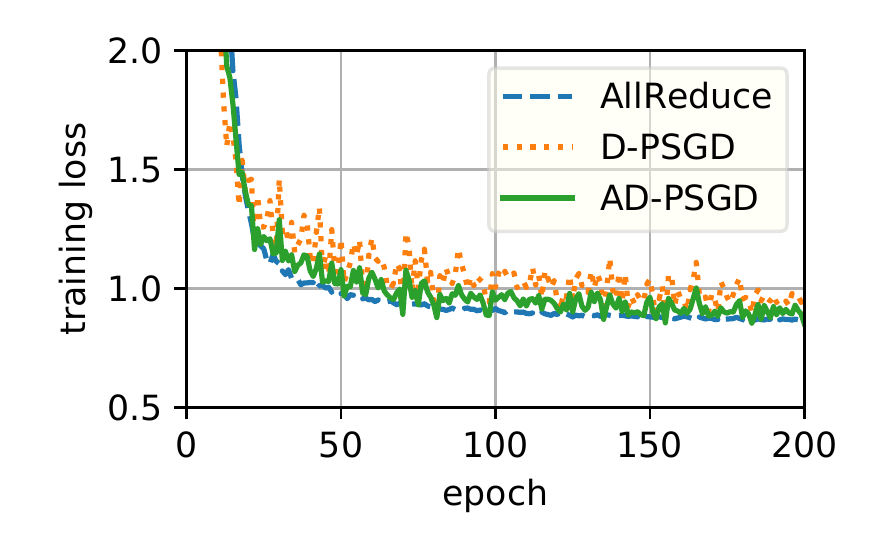} }\label{fig:nlc-jewel}}
    \subfloat[{\small NLC-Yelp}]{{ \includegraphics[width=0.5\columnwidth]{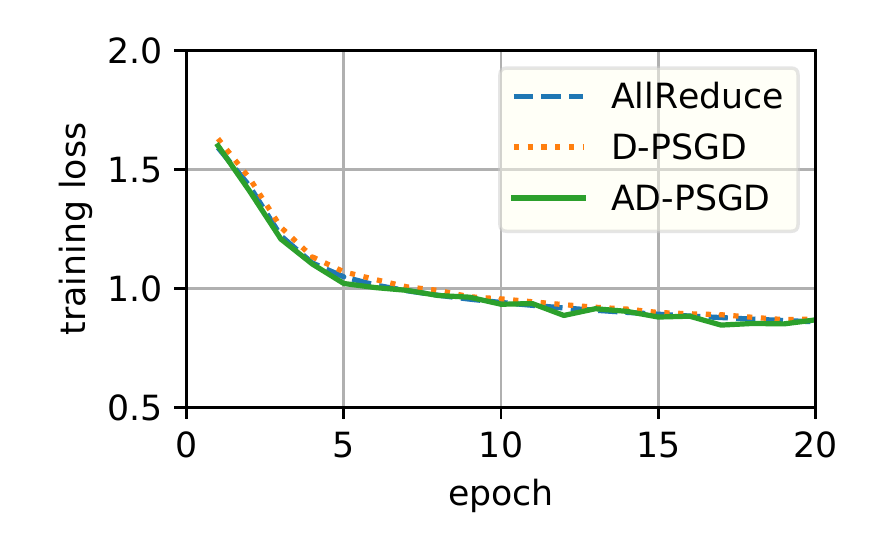}}\label{fig:yelp-jewel}}

    \caption{\small Training loss comparison for IBM NLC model on Joule and Yelp
      datasets. AllReduce-SGD, D-PSGD and AD-PSGD converge alike w.r.t. epochs.}
    \label{fig:NLC}
\end{figure}
}{\begin{figure*}[t]
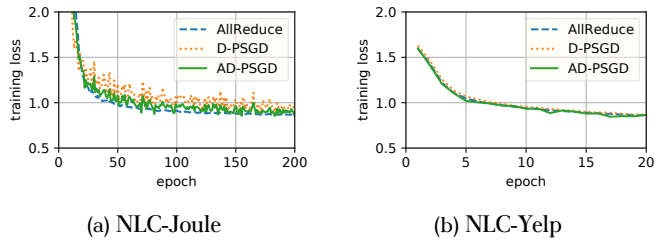
\small
    \centering
    \subfloat[{\small NLC-Joule}]{{ \includegraphics[scale=0.5]{./figures/plot/nlc-jewel} }\label{fig:nlc-jewel}}
    \subfloat[{\small NLC-Yelp}]{{ \includegraphics[scale=0.5]{./figures/plot/nlc-yelp}}\label{fig:yelp-jewel}}

    \caption{\small Training loss comparison for IBM NLC model on Joule and Yelp
      datasets. AllReduce-SGD, D-PSGD and AD-PSGD converge alike w.r.t. epochs.}
    \label{fig:NLC}
\end{figure*}
}

The IBM NLC task is to classify input sentences into a target category in a
predefined label set. The NLC model is a CNN model that has a word-embedding lookup table layer, a
convolutional layer and a fully connected layer with a softmax output layer. We
use two datasets in our evaluation. The first dataset Joule is an in-house
customer dataset that has 2.5K training samples, 1K test samples, and 311
different classes. The second dataset Yelp, which is a public dataset, has 500K
training samples, 2K test samples and 5 different classes. \Cref{fig:NLC} shows
that AD-PSGD converges (w.r.t epochs) similarly to AllReduce-SGD and D-PSGD on
NLC tasks.

{\em
  Above results show AD-PSGD converges similarly (w.r.t) to AllReduce-SGD and
  D-PSGD for IBM NLC workload, which is an example of proprietary workloads.
}

\onecolumn
\section{Appendix: proofs}

In the following analysis we define
\begin{equation}
  M_k := \sum_{i=1}^n p_{i} \left\| \frac{X_{k} \mathbf{1}_n}{n} - X_{k} e_i \right\|^2,
\end{equation}
and
\begin{equation}
  \hat{M}_k := M_{k-\tau_{k}}.
\end{equation}

We also define
\begin{align*}
  \partial f (X_k) := & n \left[\begin{array}{cccc}
    p_{1} \nabla f_1 (x_{k}^{1}) & p_{2} \nabla f_2 (x_{k}^{2}) & \cdots & p_{n}\nabla f_n (x_{k}^{n})
  \end{array}\right] \in \mathbb{R}^{N \times n} ,\\
  \partial f (X_k, i) := & \left[\begin{array}{ccccc}
    0 & \cdots & \nabla f_i (x_{k}^{i}) & \cdots & 0
  \end{array}\right] \in \mathbb{R}^{N \times n} ,\\
 \partial g (\hat{X}_{k}, \xi_k) :=& n \left[\begin{array}{cccc}
     p_1  \sum_{j = 1}^M \nabla F (\hat{x}_{k}^{1}, \xi_{k,j}^{1}) & \cdots & p_n
     \sum_{j = 1}^M \nabla F (\hat{x}_{k}^{n}, \xi_{k, j}^n)
   \end{array}\right] \in \mathbb{R}^{N\times n} .
\end{align*}

\begin{align*}
  \bar{\rho} := & \frac{n-1}{n}\left(  \frac{1}{1 - \rho} + \frac{2 \sqrt{\rho}}{\left( 1 - \sqrt{\rho} \right)^2}  \right),\\
  C_1 := & 1 - 24 M^{2} L^2 \gamma^2 \left( T \frac{n - 1}{n} + \bar{\rho} \right),\\
  C_2 := & \frac{\gamma M}{2 n} - \frac{\gamma^2 LM^2}{n^2} - \frac{2 M^3 L^2
           T^2 \gamma^3}{n^3} - \left( \frac{6 \gamma^2 L^3 M^2}{n^2} + \frac{\gamma
           M}{n} L^2 + \frac{12 M^3 L^4 T^2 \gamma^3}{n^3} \right)  \frac{4 M^2
           \gamma^2 (T\frac{n-1}{n}+ \bar{\rho})}{C_1},\\
  C_3 := & \frac{1}{2} + \frac{2 \left( 6 \gamma^2 L^2 M^2 + \gamma nML +
           \frac{12 M^3 L^3 T^2 \gamma^3}{n} \right)  \bar{\rho}}{C_1} +
           \frac{LT^2 \gamma M}{n} .
\end{align*}

\begin{proof}[Proof to \Cref{thm:main-thm}]
  We start from
  \begin{align*}
  & \mathbb{E} f \left( \frac{X_{k + 1}  \mathbf{1}_n}{n} \right)\\
  = & \mathbb{E} f \left( \frac{X_k W_k  \mathbf{1}_n}{n} - \gamma
  \frac{\partial g (\hat{X}_{k} ; \xi_{k}^{i_{k}}, i_k) \mathbf{1}_n}{n}
  \right) = \mathbb{E} f \left( \frac{X_k  \mathbf{1}_n}{n} - \gamma
  \frac{\partial g (\hat{X}_{k} ; \xi_{k}^{i_{k}}, i_k) \mathbf{1}_n}{n}
  \right)\\
  \leqslant & \mathbb{E} f \left( \frac{X_k  \mathbf{1}_n}{n} \right) - \gamma
  \mathbb{E} \left\langle \nabla f \left( \frac{X_k  \mathbf{1}_n}{n} \right),
  \frac{\partial g (\hat{X}_{k} ; \xi_{k}^{i_{k}}, i_k) \mathbf{1}_n}{n}
  \right\rangle + \frac{\gamma^2 L}{2}  \mathbb{E} \left\| \frac{\partial g
  (\hat{X}_{k} ; \xi_{k}^{i_{k}}, i_k) \mathbf{1}_n}{n} \right\|^2\\
  \overset{\text{(\ref{eq:azjlzj}), {\Cref{lemma:zhlkdsamfls}}}}{\leqslant} &
  \mathbb{E} f \left( \frac{X_k  \mathbf{1}_n}{n} \right) - \frac{\gamma M}{n}
  \mathbb{E} \left\langle \nabla f \left( \frac{X_k  \mathbf{1}_n}{n}
  \right), \frac{\partial f (\hat{X}_{k}) \mathbf{1}_n}{n} \right\rangle +
  \frac{\gamma^2 L \sigma^2 M}{2 n^2} + \frac{\gamma^2 LM^2}{2 n^2}  \sum_{i =
  1}^n p_i  \mathbb{E} \| \nabla f_i (\hat{x}_{k}^{i})\|^2\\
  = & \mathbb{E} f \left( \frac{X_k  \mathbf{1}_n}{n} \right) + \frac{\gamma
  M}{2 n}  \left\| \nabla f \left( \frac{X_k  \mathbf{1}_n}{n} \right) -
  \frac{\partial f (\hat{X}_{k}) \mathbf{1}_n}{n} \right\|^2 - \frac{\gamma
  M}{2 n}  \mathbb{E} \left\| \nabla f \left( \frac{X_k  \mathbf{1}_n}{n}
  \right) \right\|^2 - \frac{\gamma M}{2 n}  \mathbb{E} \left\| \frac{\partial
  f (\hat{X}_{k}) \mathbf{1}_n}{n} \right\|^2\\
  & + \frac{\gamma^2 LM^2}{2 n^2}  \sum_{i = 1}^n p_i  \mathbb{E} \| \nabla
  f_i (\hat{x}_{k}^{i})\|^2 + \frac{\gamma^2 L \sigma^2 M}{2 n^2} .
\end{align*}

Using the upper bound of $\sum_{i = 1}^n p_i \mathbb{E} \| \nabla f_i (\hat{x}_{k}^{i})\|^2$ in \Cref{lemma:20170614-001358}:
\begin{align*}
  & \mathbb{E} f \left( \frac{X_{k + 1}  \mathbf{1}_n}{n} \right)\\
  \leqslant & \mathbb{E} f \left( \frac{X_k  \mathbf{1}_n}{n} \right) +
  \frac{\gamma M}{2 n}  \mathbb{E}  \left\| \nabla f \left( \frac{X_k
  \mathbf{1}_n}{n} \right) - \frac{\partial f (\hat{X}_{k})
  \mathbf{1}_n}{n} \right\|^2 - \frac{\gamma M}{2 n}  \mathbb{E} \left\|
  \nabla f \left( \frac{X_k  \mathbf{1}_n}{n} \right) \right\|^2 -
  \frac{\gamma M}{2 n}  \mathbb{E}\left \| \frac{\partial f (\hat{X}_{k}) \mathbf{1}_n}{n} \right\|^2\\
  & + \frac{\gamma^2 LM^2}{2 n^2}  \left( 12 L^2 \hat{M}_{k} + 6
  \varsigma^2 + 2 \sum_{i = 1}^n p_i  \mathbb{E} \left\| \sum_{j = 1}^n p_j
  \nabla f_j (\hat{x}_{k}^{j}) \right\|^2 \right) + \frac{\gamma^2 L
  \sigma^2 M}{2 n^2}\\
  = & \mathbb{E} f \left( \frac{X_k  \mathbf{1}_n}{n} \right) + \frac{\gamma
  M}{2 n} \underbrace{\mathbb{E}  \left\| \nabla f \left( \frac{X_k
  \mathbf{1}_n}{n} \right) - \frac{\partial f (\hat{X}_{k})
  \mathbf{1}_n}{n} \right\|^2}_{T_1} - \frac{\gamma M}{2 n}  \mathbb{E}
  \left\| \nabla f \left( \frac{X_k  \mathbf{1}_n}{n} \right) \right\|^2
  \\
  & - \left( \frac{\gamma M}{2 n} - \frac{\gamma^2 LM^2}{n^2} \right)  \mathbb{E}
  \left\| \frac{\partial f (\hat{X}_{k}) \mathbf{1}_n}{n} \right\|^2
   + \frac{\gamma^2 L (\sigma^2 M + 6 \varsigma^2 M^2)}{2 n^2} + \frac{6
  \gamma^2 L^3 M^2}{n^2} \hat{M}_{k} .\numberthis \label{eq:Wed Jul 19 16:40:39 EDT 2017}
\end{align*}

For $T_1$ we have
\begin{align*}
  T_1= & \mathbb{E} \left\| \nabla f \left( \frac{X_k
  \mathbf{1}_n}{n} \right) - \frac{\partial f (\hat{X}_{k})
  \mathbf{1}_n}{n} \right\|^2\\
  \leqslant & 2\mathbb{E} \left\| \nabla f \left( \frac{X_k
  \mathbf{1}_n}{n} \right) - \nabla f \left( \frac{\hat{X}_{k}
  \mathbf{1}_n}{n} \right) \right\|^2
   + 2\mathbb{E} \left\| \nabla f \left( \frac{\hat{X}_{k}
  \mathbf{1}_n}{n} \right) - \frac{\partial f (\hat{X}_{k})
  \mathbf{1}_n}{n} \right\|^2\\
  = & 2\mathbb{E} \left\| \nabla f \left( \frac{X_k \mathbf{1}_n}{n}
  \right) - \nabla f \left( \frac{\hat{X}_{k} \mathbf{1}_n}{n} \right)
  \right\|^2 + 2 \mathbb{E} \left\| \sum_ip_{i} \left( \nabla f_i
  \left( \frac{\hat{X}_{k} \mathbf{1}_n}{n} \right) - \nabla f_i (\hat{x}_{k}^{i}) \right) \right\|^2\\
  \leqslant & 2\mathbb{E} \left\| \nabla f \left( \frac{X_k \mathbf{1}_n}{n}
  \right) - \nabla f \left( \frac{\hat{X}_{k} \mathbf{1}_n}{n} \right)
              \right\|^2 + 2 \mathbb{E}\sum_ip_{i} \left\|   \nabla f_i
              \left( \frac{\hat{X}_{k} \mathbf{1}_n}{n} \right) - \nabla f_i (\hat{x}_{k}^{i})  \right\|^2\\
  \overset{
\Cref{ass:20170614-002451}:1 }{\leqslant} & 2 L^2
  \mathbb{E}\left\| \frac{(X_k - \hat{X}_{k}) \mathbf{1}_n}{n} \right\|^2 + 2
  L^2\mathbb{E} \hat{M}_{k} .\numberthis\label{eq:Wed Jul 19 16:41:07 EDT 2017}
\end{align*}

From \eqref{eq:Wed Jul 19 16:40:39 EDT 2017} and \eqref{eq:Wed Jul 19 16:41:07
  EDT 2017} we obtain
\begin{align*}
  \mathbb{E} f \left( \frac{X_{k + 1}  \mathbf{1}_n}{n} \right) \leqslant &
  \mathbb{E} f \left( \frac{X_k  \mathbf{1}_n}{n} \right) + \frac{\gamma M}{2
  n}  \mathbb{E}  \left( 2 L^2 \left\| \frac{(X_k - \hat{X}_{k})
  \mathbf{1}_n}{n} \right\|^2 + 2 L^2 \hat{M}_{k} \right)\\
  & - \frac{\gamma M}{2 n}  \mathbb{E} \left\| \nabla f \left( \frac{X_k
  \mathbf{1}_n}{n} \right) \right\|^2 - \left( \frac{\gamma M}{2 n} -
  \frac{\gamma^2 LM^2}{n^2} \right)  \mathbb{E}\left \| \frac{\partial f (\hat{X}_{k}) \mathbf{1}_n}{n} \right\|^2\\
  & + \frac{6 \gamma^2 L^3 M^2}{n^2}  \mathbb{E} \hat{M}_{k} +
  \frac{\gamma^2 L (\sigma^2 M + 6 \varsigma^2 M^2)}{2 n^2}\\
  = & \mathbb{E} f \left( \frac{X_k  \mathbf{1}_n}{n} \right) - \frac{\gamma
  M}{2 n}  \mathbb{E} \left\| \nabla f \left( \frac{X_k  \mathbf{1}_n}{n}
  \right) \right\|^2 - \left( \frac{\gamma M}{2 n} - \frac{\gamma^2 LM^2}{n^2}
  \right)  \mathbb{E}\left \| \frac{\partial f (\hat{X}_{k}) \mathbf{1}_n}{n} \right\|^2\\
  & + \left( \frac{6 \gamma^2 L^3 M^2}{n^2} + \frac{\gamma M}{n} L^2 \right)
  \mathbb{E} \hat{M}_{k} + \frac{\gamma M}{n} L^2  \mathbb{E} \left\|
  \frac{(X_k - \hat{X}_{k}) \mathbf{1}_n}{n} \right\|^2 + \frac{\gamma^2 L
  (\sigma^2 M + 6 \varsigma^2 M^2)}{2 n^2}\\
  \overset{\text{{\Cref{lemma:20170614-115656}}}}{\leqslant} & \mathbb{E} f
  \left( \frac{X_k  \mathbf{1}_n}{n} \right) - \frac{\gamma M}{2 n}
  \mathbb{E} \left\| \nabla f \left( \frac{X_k  \mathbf{1}_n}{n} \right)
  \right\|^2 - \left( \frac{\gamma M}{2 n} - \frac{\gamma^2 LM^2}{n^2} \right)
  \mathbb{E} \left\| \frac{\partial f (\hat{X}_{k}) \mathbf{1}_n}{n} \right\|^2\\
  & + \left( \frac{6 \gamma^2 L^3 M^2}{n^2} + \frac{\gamma M}{n} L^2 \right)
  \mathbb{E} \hat{M}_{k} + \frac{\gamma M}{n} L^2  \left( \frac{\tau_k^2
  \gamma^2 \sigma^2 M}{n^2} + \tau_k \gamma^2  \sum_{t = 1}^{\tau_k} \left(
  \frac{M^2}{n^2}  \sum_{i = 1}^n p_i  \mathbb{E} \| \nabla f_i (\hat{x}_{k - t}^{i})\|^2 \right) \right)\\
  & + \frac{\gamma^2 L (\sigma^2 M + 6 \varsigma^2 M^2)}{2 n^2}\\
  \leqslant & \mathbb{E} f \left( \frac{X_k  \mathbf{1}_n}{n} \right) -
  \frac{\gamma M}{2 n}  \mathbb{E} \left\| \nabla f \left( \frac{X_k
  \mathbf{1}_n}{n} \right) \right\|^2 - \left( \frac{\gamma M}{2 n} -
  \frac{\gamma^2 LM^2}{n^2} \right)  \mathbb{E}\left \| \frac{\partial f (\hat{X}_{k}) \mathbf{1}_n}{n} \right\|^2\\
  & + \left( \frac{6 \gamma^2 L^3 M^2}{n^2} + \frac{\gamma M}{n} L^2 \right)
  \mathbb{E} \hat{M}_{k} + \frac{\gamma^2 L (\sigma^2 M + 6 \varsigma^2
  M^2)}{2 n^2} + \frac{L^2 T^2 \gamma^3 \sigma^2 M^2}{n^3}\\
  & + \frac{M^3 L^2 \tau_k \gamma^3}{n^3}  \sum_{t = 1}^{\tau_k} \left(
  \sum_{i = 1}^n p_i  \mathbb{E} \| \nabla f_i (\hat{x}_{k - t}^{i})\|^2 \right)\\
  \overset{\text{{\Cref{lemma:20170614-001358}}}}{\leqslant} & \mathbb{E} f
  \left( \frac{X_k  \mathbf{1}_n}{n} \right) - \frac{\gamma M}{2 n}
  \mathbb{E} \left\| \nabla f \left( \frac{X_k  \mathbf{1}_n}{n} \right)
  \right\|^2 - \left( \frac{\gamma M}{2 n} - \frac{\gamma^2 LM^2}{n^2} \right)
  \mathbb{E} \left\| \frac{\partial f (\hat{X}_{k}) \mathbf{1}_n}{n} \right\|^2\\
  & + \left( \frac{6 \gamma^2 L^3 M^2}{n^2} + \frac{\gamma M}{n} L^2 \right)
  \mathbb{E} \hat{M}_{k} + \frac{\gamma^2 L (\sigma^2 M + 6 \varsigma^2
  M^2)}{2 n^2} + \frac{L^2 T^2 \gamma^3 \sigma^2 M^2}{n^3}\\
  & + \frac{M^3 L^2 \tau_k \gamma^3}{n^3}  \sum_{t = 1}^{\tau_k} \left( 12
  L^2 \hat{M}_{k - t } + 6 \varsigma^2 + 2 \mathbb{E} \left\| \sum_{j
  = 1}^n p_j \nabla f_j (\hat{x}_{k - t}^{j}) \right\|^2 \right)\\
  =& \mathbb{E} f \left( \frac{X_k  \mathbf{1}_n}{n} \right) - \frac{\gamma
  M}{2 n}  \mathbb{E} \left\| \nabla f \left( \frac{X_k  \mathbf{1}_n}{n}
  \right) \right\|^2 - \left( \frac{\gamma M}{2 n} - \frac{\gamma^2 LM^2}{n^2}
  \right)  \mathbb{E} \left\| \frac{\partial f (\hat{X}_{k}) \mathbf{1}_n}{n} \right\|^2\\
  & + \left( \frac{6 \gamma^2 L^3 M^2}{n^2} + \frac{\gamma M}{n} L^2 \right)
  \mathbb{E} \hat{M}_{k} + \frac{\gamma^2 L (\sigma^2 M + 6 \varsigma^2
  M^2)}{2 n^2} + \frac{L^2 T^2 \gamma^3 M (\sigma^2 M + 6 \varsigma^2
  M^2)}{n^3}\\
  & + \frac{2 M^3 L^2 T \gamma^3}{n^3}  \sum_{t = 1}^{\tau_k} \left( 6
  L^2  \mathbb{E} \hat{M}_{k - t} + \mathbb{E} \left\| \frac{\partial
  f (\hat{X}_{k - t}) \mathbf{1}_n}{n} \right\|^2 \right) .
\end{align*}

Summing from $k = 0$ to $k = K - 1$ we obtain
\begin{align*}
  \mathbb{E} f \left( \frac{X_K  \mathbf{1}_n}{n} \right) \leqslant &
  \mathbb{E} f \left( \frac{X_0  \mathbf{1}_n}{n} \right) - \frac{\gamma M}{2
  n}  \sum_{k = 0}^{K - 1} \mathbb{E} \left\| \nabla f \left( \frac{X_k
  \mathbf{1}_n}{n} \right) \right\|^2 - \left( \frac{\gamma M}{2 n} -
  \frac{\gamma^2 LM^2}{n^2} \right)  \sum_{k = 0}^{K - 1} \mathbb{E} \left\|
  \frac{\partial f (\hat{X}_{k}) \mathbf{1}_n}{n} \right\|^2\\
  & + \left( \frac{6 \gamma^2 L^3 M^2}{n^2} + \frac{\gamma M}{n} L^2 \right)
  \sum_{k = 0}^{K - 1} \mathbb{E} \hat{M}_{k} + \frac{\gamma^2 L (\sigma^2
  M + 6 \varsigma^2 M^2) K}{2 n^2} + \frac{L^2 T^2 \gamma^3 M (\sigma^2 M + 6
  \varsigma^2 M^2) K}{n^3}\\
  & + \frac{2 M^3 L^2 T \gamma^3}{n^3}  \sum_{k = 0}^{K - 1} \sum_{t =
  1}^{\tau_k} \left( 6 L^2  \mathbb{E} \hat{M}_{k - t} + \mathbb{E}
  \left\| \frac{\partial f (\hat{X}_{k - t}) \mathbf{1}_n}{n}
  \right\|^2 \right)\\
  \leqslant & \mathbb{E} f \left( \frac{X_0  \mathbf{1}_n}{n} \right) -
  \frac{\gamma M}{2 n}  \sum_{k = 0}^{K - 1} \mathbb{E} \left\| \nabla f
  \left( \frac{X_k  \mathbf{1}_n}{n} \right) \right\|^2 - \left( \frac{\gamma
  M}{2 n} - \frac{\gamma^2 LM^2}{n^2} \right)  \sum_{k = 0}^{K - 1} \mathbb{E}
  \| \frac{\partial f (\hat{X}_{k}) \mathbf{1}_n}{n} \|^2\\
  & + \left( \frac{6 \gamma^2 L^3 M^2}{n^2} + \frac{\gamma M}{n} L^2 \right)
  \sum_{k = 0}^{K - 1} \mathbb{E} \hat{M}_{k} + \frac{\gamma^2 L (\sigma^2
  M + 6 \varsigma^2 M^2) K}{2 n^2} + \frac{L^2 T^2 \gamma^3 M (\sigma^2 M + 6
  \varsigma^2 M^2) K}{n^3}\\
  & + \frac{2 M^3 L^2 T^2 \gamma^3}{n^3}  \sum_{k = 0}^{K - 1} \left( 6 L^2
  \mathbb{E} \hat{M}_{k} + \mathbb{E} \left\| \frac{\partial f (\hat{X}_{k}) \mathbf{1}_n}{n} \right\|^2 \right)\\
  = & \mathbb{E} f \left( \frac{X_0  \mathbf{1}_n}{n} \right) - \frac{\gamma
  M}{2 n}  \sum_{k = 0}^{K - 1} \mathbb{E} \left\| \nabla f \left( \frac{X_k
  \mathbf{1}_n}{n} \right) \right\|^2\\
  & - \left( \frac{\gamma M}{2 n} - \frac{\gamma^2 LM^2}{n^2} - \frac{2 M^3
  L^2 T^2 \gamma^3}{n^3} \right)  \sum_{k = 0}^{K - 1} \mathbb{E} \left\|
  \frac{\partial f (\hat{X}_{k}) \mathbf{1}_n}{n} \right\|^2\\
  & + \left( \frac{6 \gamma^2 L^3 M^2}{n^2} + \frac{\gamma M}{n} L^2 +
  \frac{12 M^3 L^4 T^2 \gamma^3}{n^3} \right)  \sum_{k = 0}^{K - 1} \mathbb{E}
  \hat{M}_{k}\\
  & + \frac{\gamma^2 L (\sigma^2 M + 6 \varsigma^2 M^2) K}{2 n^2} + \frac{L^2
  T^2 \gamma^3 M (\sigma^2 M + 6 \varsigma^2 M^2) K}{n^3} .\\
  \leqslant & \mathbb{E} f \left( \frac{X_0  \mathbf{1}_n}{n} \right) -
  \frac{\gamma M}{2 n}  \sum_{k = 0}^{K - 1} \mathbb{E} \left\| \nabla f
  \left( \frac{X_k  \mathbf{1}_n}{n} \right) \right\|^2\\
  & - \left( \frac{\gamma M}{2 n} - \frac{\gamma^2 LM^2}{n^2} - \frac{2 M^3
  L^2 T^2 \gamma^3}{n^3} \right)  \sum_{k = 0}^{K - 1} \mathbb{E} \left\|
  \frac{\partial f (\hat{X}_{k}) \mathbf{1}_n}{n} \right\|^2\\
  & + \left( \frac{6 \gamma^2 L^3 M^2}{n^2} + \frac{\gamma M}{n} L^2 +
  \frac{12 M^3 L^4 T^2 \gamma^3}{n^3} \right)  \sum_{k = 0}^{K - 1} \mathbb{E}
  \hat{M}_{k}\\
  & + \frac{\gamma^2 L (\sigma^2 M + 6 \varsigma^2 M^2) K}{2 n^2} + \frac{L^2
  T^2 \gamma^3 M (\sigma^2 M + 6 \varsigma^2 M^2) K}{n^3}\\
  \overset{C_1 > 0, \text{{\Cref{lemma:20170614-114822}}}}{\leqslant} &
  \mathbb{E} f \left( \frac{X_0  \mathbf{1}_n}{n} \right) - \frac{\gamma M}{2
  n}  \sum_{k = 0}^{K - 1} \mathbb{E} \left\| \nabla f \left( \frac{X_k
  \mathbf{1}_n}{n} \right) \right\|^2\\
  & - \left( \frac{\gamma M}{2 n} - \frac{\gamma^2 LM^2}{n^2} - \frac{2 M^3
  L^2 T^2 \gamma^3}{n^3} \right)  \sum_{k = 0}^{K - 1} \mathbb{E} \left\|
  \frac{\partial f (\hat{X}_{k}) \mathbf{1}_n}{n} \right\|^2\\
  & + \left( \frac{6 \gamma^2 L^3 M^2}{n^2} + \frac{\gamma M}{n} L^2 +
  \frac{12 M^3 L^4 T^2 \gamma^3}{n^3} \right) K \frac{2 \gamma^2  (M \sigma^2
  + 6 M^2 \varsigma^2)  \bar{\rho}}{C_1}\\
  & + \left( \frac{6 \gamma^2 L^3 M^2}{n^2} + \frac{\gamma M}{n} L^2 +
  \frac{12 M^3 L^4 T^2 \gamma^3}{n^3} \right)  \frac{4 M^2 \gamma^2  \left( T
  \frac{n - 1}{n}  + \bar{\rho} \right) \sum_{k = 0}^{K - 1} \mathbb{E} \left\|
  \sum_{i = 1}^n p_i \nabla f_i (\hat{x}_{k}^{i}) \right\|^2}{C_1}\\
  & + \frac{\gamma^2 L (\sigma^2 M + 6 \varsigma^2 M^2) K}{2 n^2} + \frac{L^2
  T^2 \gamma^3 M (\sigma^2 M + 6 \varsigma^2 M^2) K}{n^3}\\
  = & \mathbb{E} f \left( \frac{X_0  \mathbf{1}_n}{n} \right) - \frac{\gamma
  M}{2 n}  \sum_{k = 0}^{K - 1} \mathbb{E} \left\| \nabla f \left( \frac{X_k
  \mathbf{1}_n}{n} \right) \right\|^2\\
  & - C_2  \sum_{k = 0}^{K - 1} \mathbb{E} \left\| \frac{\partial f (\hat{X}_{k}) \mathbf{1}_n}{n} \right\|^2 + C_3  \frac{\gamma^2 LK}{n^2}  (M
  \sigma^2 + 6 M^2 \varsigma^2) .
\end{align*}
Thus while $C_3 \leqslant 1$ and $C_2 \geqslant 0$ we have
\begin{align*}
  \frac{\sum_{k = 0}^{K - 1} \mathbb{E} \left\| \nabla f \left( \frac{X_k
  \mathbf{1}_n}{n} \right) \right\|^2}{K} \leqslant & \frac{2 \left(
  \mathbb{E} f \left( \frac{X_0  \mathbf{1}_n}{n} \right) - \mathbb{E} f
  \left( \frac{X_K  \mathbf{1}_n}{n} \right) \right)}{\gamma KM / n} + \frac{2
  \gamma L}{Mn}  (M \sigma^2 + 6 M^2 \varsigma^2)\\
  \leqslant & \frac{2 (\mathbb{E} f (x_0) - \mathbb{E} f^{\ast})}{\gamma KM /
  n} + \frac{2 \gamma L}{n}  ( \sigma^2 + 6 M \varsigma^2) .
\end{align*}
It completes the proof.
\end{proof}

\begin{lemma}\label{lemma:zhzlxdfmlqwr}
Define $\prod_{k = 1}^0 W_k = I$, where $I$ is the identity matrix. Then
\begin{align*}
  \mathbb{E}\left\|\frac{\mathbf{1}_n}{n} - \prod_{k = 1}^K W_k e_i\right \|^2 \le \frac{n-1}{n}\rho^K, \quad \forall K \ge 0.
\end{align*}
\end{lemma}
\begin{proof}
  Let $y_K = \frac{\mathbf{1}_n}{n} - \prod_{k = 1}^K W_k e_i$. Then noting that
  $y_{K+1} = W_{K+1}y_{K}$ we have
\begin{align*}
  & \mathbb{E} \| y_{K + 1} \|^2\\
  = & \mathbb{E} \| W_{k+1} y_{K}\|^{2}\\
  = & \mathbb{E} \langle W_{K + 1} y_K, W_{K + 1} y_K \rangle\\
  = & \mathbb{E} \langle y_K, W^{\top}_{K + 1} W_{K + 1} y_K \rangle \\
  = & \mathbb{E} \langle y_K, \mathbb{E}\mathbb{E}_{i_{K+1}}(W^{\top}_{K + 1} W_{K + 1}) y_K \rangle\\
  = & \mathbb{E} \langle y_K, \mathbb{E}(W^{\top}_{K + 1} W_{K + 1}) y_K \rangle.
\end{align*}

Note that $\mathbb{E}(W^{\top}_{K + 1} W_{K + 1})$ is symmetric and doubly
stochastic and $\mathbf{1}_{n}$ is an eigenvector of $\mathbb{E}(W^{\top}_{K +
  1} W_{K + 1})$ with eigenvalue 1. Starting from $\mathbf{1}_{n}$ we construct
a basis of $\mathbb{R}^{n}$ composed by the eigenvectors of
$\mathbb{E}(W^{\top}_{K + 1} W_{K + 1})$, which is guaranteed to exist by the
spectral theorem of Hermitian matrices. From \eqref{eq:2lkzjsdf} the magnitude
of all other eigenvectors' associated eigenvalues should be smaller or equal to
$\rho$. Noting $y_{K}$ is orthogonal to $\mathbf{1}_{n}$, we decompose $y_{K}$
using this constructed basis and it follows that
\begin{align*}
 \mathbb{E} \| y_{K + 1} \|^2 \leqslant & \rho \mathbb{E} \| y_K \|^2 .
\end{align*}
Noting that $\|y_0 \|^2 = \| \mathbf{1}_n / n - e_i \|^2 = \frac{(n - 1)^2}{n^2} + \sum_{i
= 1}^{n - 1} \frac{1}{n^2} = \frac{n^2 - 2 n + 1 + n - 1}{n^2} = \frac{n -
1}{n}$, by induction, we complete the proof.
\end{proof}

\begin{lemma}\label{lemma:zhlkdsamfls}
  \[ \mathbb{E} \left\| \frac{\partial g (\hat{X}_{k} ; \xi_{k}^{i_{k}}, i_k)
   \mathbf{1}_n}{n} \right\|^2 \leqslant \frac{\sigma^2 M}{n^2} +
   \frac{M^2}{n^2}  \sum_{i = 1}^n p_i  \mathbb{E} \| \nabla f_i (\hat{x}_{k}^{i})\|^2 ,\quad \forall k\ge 0. \]
\end{lemma}

\begin{proof}The LHS can be bounded by
  \begin{align*}
  \mathbb{E} \left\| \frac{\partial g (\hat{X}_{k} ; \xi_{k}^{i_{k}}, i_k)
  \mathbf{1}_n}{n} \right\|^2 \overset{(\ref{eq:4zljdslaf})}{=} & \sum_{i =
  1}^n p_i  \mathbb{E} \left\| \frac{\sum_{j = 1}^M \nabla F (\hat{x}_{k}^{i}, \xi_{k, j}^i)}{n} \right\|^2\\
  = & \sum_{i = 1}^n p_i  \mathbb{E} \left\| \frac{\sum_{j = 1}^M (\nabla F
  (\hat{x}_{k}^{i}, \xi_{k, j}^i) - \nabla f_i (\hat{x}_{k}^{i}))}{n}
  \right\|^2 + \sum_{i = 1}^n p_i  \mathbb{E} \left\| \frac{M \nabla f_i (\hat{x}_{k}^{i})}{n} \right\|^2\\
  \overset{(\ref{eq:20170613-221110})}{\leqslant} & \frac{\sigma^2 M}{n^2} +
  \frac{M^2}{n^2}  \sum_{i = 1}^n p_i  \mathbb{E} \| \nabla f_i (\hat{x}_{k}^{i})\|^2 .
\end{align*}
\end{proof}

\begin{lemma}
\label{lemma:20170614-001358}
\begin{align*}
 \sum_{i=1}^{n} p_{i} \mathbb{E} \| \nabla f_{i} (\hat{x}_{k}^{i}) \|^2 \leqslant & 12 L^2 \mathbb{E}\hat{M}_{k} + 6 \varsigma^2 + 2   \mathbb{E} \left\| \sum_{j = 1}^n p_j \nabla f_j (\hat{x}_{k}^{j}) \right\|^2 ,\quad \forall k\ge 0.
\end{align*}
\end{lemma}

\begin{proof}
The LHS can be bounded by
\begin{align*}
  \sum_{i = 1}^n p_i  \mathbb{E} \| \nabla f_i (\hat{x}_{k}^{i})\|^2
  \leqslant & \sum_{i = 1}^n p_i  \mathbb{E}  \| \nabla f_i (\hat{x}_{k}^{i}) - \sum_{j = 1}^n p_j \nabla f_j (\hat{x}_{k}^{j}) + \sum_{j = 1}^n p_j
  \nabla f_j (\hat{x}_{k}^{j})\|^2\\
  \leqslant & 2 \sum_{i = 1}^n p_i  \mathbb{E}  \left\| \nabla f_i (\hat{x}_{k}^{i}) - \sum_{j = 1}^n p_j \nabla f_j (\hat{x}_{k}^{j}) \right\|^2 +
  2 \sum_{i = 1}^n p_i  \mathbb{E} \left\| \sum_{j = 1}^n p_j \nabla f_j (\hat{x}_{k}^{j}) \right\|^2 \\
  = & 2 \sum_{i = 1}^n p_i  \mathbb{E}  \left\| \nabla f_i (\hat{x}_{k}^{i}) - \sum_{j = 1}^n p_j \nabla f_j (\hat{x}_{k}^{j}) \right\|^2 +
  2 \mathbb{E} \left\| \sum_{j = 1}^n p_j \nabla f_j (\hat{x}_{k}^{j}) \right\|^2 .\numberthis\label{eq:zjlafmsf}
\end{align*}

For the first term on the RHS we have
\begin{align*}
  & \sum_{i = 1}^n p_i  \mathbb{E}  \left\| \nabla f_i (\hat{x}_{k}^{i}) -
    \sum_{j = 1}^n p_j \nabla f_j (\hat{x}_{k}^{j}) \right\|^2\\
  \leqslant & 3 \sum_{i = 1}^n p_i  \mathbb{E}  \left\| \nabla f_i (\hat{x}_{k}^{i}) - \nabla f_i \left( \frac{\hat{X}_{k}  \mathbf{1}_n}{n}
              \right) \right\|^2 + 3 \sum_{i = 1}^n p_i  \mathbb{E} \left\| \nabla f_i
              \left( \frac{\hat{X}_{k}  \mathbf{1}_n}{n} \right) - \sum_{j = 1}^n p_j
              \nabla f_j \left( \frac{\hat{X}_{k}  \mathbf{1}_n}{n} \right)
              \right\|^2\\
  & + 3 \sum_{i = 1}^n p_i  \mathbb{E}  \left\| \sum_{j = 1}^n p_j \nabla f_j
    (\hat{x}_{k}^{j}) - \sum_{j = 1}^n p_j \nabla f_j \left( \frac{\hat{X}_{k}  \mathbf{1}_n}{n} \right) \right\|^2\\
  \leqslant & 3 L^2  \sum_{i = 1}^n p_i  \mathbb{E}  \left\| \hat{x}_{k}^{i}
              - \frac{\hat{X}_{k}  \mathbf{1}_n}{n} \right\|^2 + 3 \sum_{i = 1}^n p_i
              \mathbb{E}  \left\| \nabla f_i \left( \frac{\hat{X}_{k}  \mathbf{1}_n}{n}
              \right) - \sum_{j = 1}^n p_j \nabla f_j \left( \frac{\hat{X}_{k}
              \mathbf{1}_n}{n} \right) \right\|^2\\
  & + 3 \mathbb{E}  \left\| \sum_{j = 1}^n p_j \nabla f_j (\hat{x}_{k}^{j})
    - \sum_{j = 1}^n p_j \nabla f_j \left( \frac{\hat{X}_{k}
    \mathbf{1}_n}{n} \right) \right\|^2\\
  \leqslant & 3 L^2 \mathbb{E}\hat{M}_{k} + 3 \sum_{i = 1}^n p_i  \mathbb{E}  \left\| \nabla f_i \left( \frac{\hat{X}_{k}  \mathbf{1}_n}{n} \right) - \nabla f \left( \frac{\hat{X}_{k}
              \mathbf{1}_n}{n} \right) \right\|^2 + 3 \sum_{j = 1}^n p_j  \mathbb{E}  \left\| \nabla f_j (\hat{x}_{k}^{j}) - \nabla f_j \left( \frac{\hat{X}_{k}  \mathbf{1}_n}{n} \right)
              \right\|^2\\
  \leqslant & 6 L^2 \mathbb{E}\hat{M}_{k} + 3 \varsigma^2 .
\end{align*}
Plugging this upper bound into \eqref{eq:zjlafmsf} we complete the proof.
\end{proof}

\begin{lemma}
\label{lemma:20170614-010128}
For any $k \ge -1$ we have
  \begin{align*}
  & \mathbb{E} \left\| \frac{X_{k + 1} \mathbf{1}_n}{n} - X_{k + 1} e_i
  \right\|^2\\
    \leqslant &
  2 \gamma^2  (M \sigma^2 + 6 M^{2} \varsigma^2)  \bar{\rho}\\
  & + 2\frac{n-1}{n} M^{2} \gamma^2  \mathbb{E}  \sum_{j = 0}^k \left( 12 L^2 \hat{M}_{j} + 2
  \mathbb{E} \left\| \sum_{i = 1}^n p_i \nabla f_i (\hat{x}_{j}^{i})
  \right\|^2 \right)  \left( \rho^{k - j} + 2 (k - j) \rho^{\frac{k - j}{2}}
  \right) .
  \end{align*}
\end{lemma}

\begin{proof}
Note that for $k = -1$, we have
  \begin{align*}
  & \mathbb{E} \left\| \frac{X_{k + 1} \mathbf{1}_n}{n} - X_{k + 1} e_i
  \right\|^2 = 0.
  \end{align*}

  Also note that the columns of $X_{0}$ are the same (all workers start with the same
  model), we have $X_{0}W_{k} = X_{0}$ for all $k$ and
  $X_{0}\mathbf{1}_{n}/n - X_{0}e_{i} = 0,\forall i$. It follows that
\begin{align*}
  & \mathbb{E}  \left\| \frac{X_{k + 1}  \mathbf{1}_n}{n} - X_{k + 1} e_i
  \right\|^2\\
  = & \mathbb{E}  \left\| \frac{X_k  \mathbf{1}_n - \gamma \partial g (\hat{X}_{k} ; \xi_{k}^{i_{k}}, i_k) \mathbf{1}_n}{n} - (X_k W_k e_i - \gamma
  \partial g (\hat{X}_{k}, \xi_{k}^{i_{k}}, i_k) e_i) \right\|^2\\
  = & \mathbb{E}  \left\| \frac{X_0  \mathbf{1}_n - \sum_{j = 0}^k \gamma
  \partial g (\hat{X}_{j} ; \xi_{j}^{i_{j}}, i_j) \mathbf{1}_n}{n} - \left(
  X_0  \prod_{j = 0}^k W_j e_i - \sum_{j = 0}^k \gamma \partial g (\hat{X}_{j} ; \xi_{j}^{i_{j}}, i_j) \prod_{q = j + 1}^k W_q e_i \right)
  \right\|^2\\
  = & \mathbb{E}  \left\| - \sum_{j = 0}^k \gamma \partial
  g (\hat{X}_{j} ; \xi_{j}^{i_{j}}, i_j) \frac{\mathbf{1}_n}{n} + \sum_{j =
  0}^k \gamma \partial g (\hat{X}_{j} ; \xi_{j}^{i_{j}}, i_j) \prod_{q = j +
  1}^k W_q e_i \right\|^2\\
  = & \gamma^2  \mathbb{E} \left\|
  \sum_{j = 0}^k \partial g (\hat{X}_{j}, \xi_{j}^{i_{j}}, i_j) \left(
  \frac{\mathbf{1}_n}{n} - \prod_{q = j + 1}^k W_q e_i \right) \right\|^2\\
  \leqslant & 2 \gamma^2 \underbrace{\mathbb{E} \left\| \sum_{j = 0}^k
  (\partial g (\hat{X}_{j}, \xi_{j}^{i_{j}}, i_j) - M \partial f (\hat{X}_{j}, i_j)) \left( \frac{\mathbf{1}_n}{n} - \prod_{q = j + 1}^k W_q e_i
  \right) \right\|^2}_{A_1}\\
  & + 2 M^{2}\gamma^2 \underbrace{\mathbb{E} \left\| \sum_{j = 0}^k \partial f
  (\hat{X}_{j}, i_j) \left( \frac{\mathbf{1}_n}{n} - \prod_{q = j + 1}^k
  W_q e_i \right) \right\|^2}_{A_2} .\numberthis \label{eq:Wed Jul 19 11:16:27 EDT 2017}
\end{align*}

For $A_1$,
\begin{align*}
  A_1 = & \mathbb{E} \left\| \sum_{j = 0}^k (\partial g (\hat{X}_{j},
  \xi_{j}^{i_{j}}, i_j) -M  \partial f (\hat{X}_{j}, i_j)) \left(
  \frac{\mathbf{1}_n}{n} - \prod_{q = j + 1}^k W_q e_i \right) \right\|^2\\
  = & \underbrace{\sum_{j = 0}^k \mathbb{E} \left\| (\partial g (\hat{X}_{j}, \xi_{j}^{i_{j}}, i_j) - M\partial f (\hat{X}_{j}, i_j)) \left(
  \frac{\mathbf{1}_n}{n} - \prod_{q = j + 1}^k W_q e_i \right)
  \right\|^2}_{A_3}\\
  & + 2 \underbrace{\mathbb{E}  \sum_{k \geqslant j > j' \geqslant 0}
  \begin{array}{c}
    \left\langle (\partial g (\hat{X}_{j}, \xi_{j}^{i_{j}}, i_j) - M \partial f
    (\hat{X}_{j}, i_j)) \left( \frac{\mathbf{1}_n}{n} - \prod_{q = j + 1}^k
    W_q e_i \right), \right.\\
    \left. (\partial g (\hat{X}_{j'} \xi_{j'}^{i_{j'}}, i_{j'}) -
    M\partial f (\hat{X}_{j'}, i_{j'})) \left( \frac{\mathbf{1}_n}{n} -
    \prod_{q = j' + 1}^k W_q e_i \right) \right\rangle
  \end{array}}_{A_4} .
\end{align*}

$A_3$ can be bounded by a constant:
\begin{align*}
  A_3 = & \sum_{j = 0}^k \mathbb{E} \left\| (\partial g (\hat{X}_{j},
  \xi_{j}^{i_{j}}, i_j) - M\partial f (\hat{X}_{j}, i_j)) \left(
  \frac{\mathbf{1}_n}{n} - \prod_{q = j + 1}^{k} W_q e_i \right)
  \right\|^2\\
  \leqslant & \sum_{j = 0}^k \mathbb{E}  \| \partial g (\hat{X}_{j},
  \xi_{j}^{i_{j}}, i_j) - M\partial f (\hat{X}_{j}, i_j)\|^2  \left\|
  \frac{\mathbf{1}_n}{n} - \prod_{q = j + 1}^{k} W_q e_i \right\|^2\\
  \overset{\text{\Cref{lemma:zhzlxdfmlqwr}}}{\leqslant} & \frac{n-1}{n}\sum_{j = 0}^k \mathbb{E}  \| \partial g (\hat{X}_{j},
  \xi_{j}^{i_{j}}, i_j) - M\partial f (\hat{X}_{j}, i_j)\|^2 \rho^{k - j}\\
  \overset{\text{\Cref{ass:20170614-002451}:5}}{\leqslant} &\frac{n-1}{n} M \sigma^2  \sum_{j = 0}^k \rho^{k - j} \leqslant \frac{n-1}{n} \frac{M \sigma^2}{1 - \rho} .
\end{align*}

$A_4$ can be bounded by another constant:
\begin{align*}
  A_4 = & \sum_{k \geqslant j > j' \geqslant 0} \mathbb{E} \begin{array}{c}
    \left\langle (\partial g (\hat{X}_{j}, \xi_{j}^{i_{j}}, i_j) -M \partial f
    (\hat{X}_{j}, i_j)) \left( \frac{\mathbf{1}_n}{n} - \prod_{q = j + 1}^k
    W_q e_i \right), \right.\\
    \left. (\partial g (\hat{X}_{j'}, \xi_{j'}^{i_{j'}}, i_{j'}) -M
    \partial f (\hat{X}_{j'}, i_{j'})) \left( \frac{\mathbf{1}_n}{n} -
    \prod_{q = j' + 1}^k W_q e_i \right) \right\rangle
  \end{array}\\
  \leqslant & \sum_{k \geqslant j > j' \geqslant 0} \begin{array}{c}
    \mathbb{E}  \| \partial g (\hat{X}_{j}, \xi_{j}^{i_{j}}, i_j) - M\partial f
    (\hat{X}_{j}, i_j)\|  \left\| \frac{\mathbf{1}_n}{n} - \prod_{q = j +
    1}^k W_q e_i \right\| \times\\
    \| \partial g (\hat{X}_{j'}, \xi_{j'}^{i_{j'}}, i_{j'}) - M\partial f
    (\hat{X}_{j'}, i_{j'})\|  \left\| \frac{\mathbf{1}_n}{n} - \prod_{q
    = j' + 1}^k W_q e_i \right\|
  \end{array}\\
  \leqslant & \mathbb{E}  \sum_{k \geqslant j > j' \geqslant 0} \left(
  \begin{array}{c}
    \frac{\left\| \frac{\mathbf{1}_n}{n} - \prod_{q = j + 1}^k W_q e_i
    \right\|^2  \left\| \frac{\mathbf{1}_n}{n} - \prod_{q = j' + 1}^k W_q e_i
    \right\|^2}{2 \alpha_{j, j'}}\\
    + \frac{\| \partial g (\hat{X}_{j}, \xi_{j}^{i_{j}}, i_j) - M\partial f
    (\hat{X}_{j}, i_j)\|^2 \| \partial g (\hat{X}_{j'}, \xi_{j'}^{i_{j'}}, i_{j'}) -M \partial f (\hat{X}_{j'}, i_{j'})\|^2}{2 /
    \alpha_{j, j'}}
  \end{array} \right), \forall \alpha_{j, j'} > 0\\
  \overset{(\ref{eq:20170613-221110})}{\leqslant} & \mathbb{E}  \sum_{k
  \geqslant j > j' \geqslant 0} \left( \frac{\left\| \frac{\mathbf{1}_n}{n} -
  \prod_{q = j + 1}^k W_q e_i \right\|^2  \left\| \frac{\mathbf{1}_n}{n} -
  \prod_{q = j' + 1}^k W_q e_i \right\|^2}{2 \alpha_{j, j'}} + \frac{M^2
  \sigma^4}{2 / \alpha_{j, j'}} \right), \forall \alpha_{j, j'} > 0\\
  \leqslant & \mathbb{E}  \sum_{k \geqslant j > j' \geqslant 0} \left(\frac{n-1}{n}
  \frac{\left\| \frac{\mathbf{1}_n}{n} - \prod_{q = j' + 1}^k W_q e_i
  \right\|^2}{2 \alpha_{j, j'}} + \frac{M^2 \sigma^4}{2 / \alpha_{j, j'}}
  \right), \forall \alpha_{j, j'} > 0\\
  \overset{\text{\Cref{lemma:zhzlxdfmlqwr}}}{\leqslant} & \mathbb{E}  \sum_{k \geqslant j > j' \geqslant 0} \left(
  \left(  \frac{n-1}{n}\right)^{2}\frac{\rho^{k - j'}}{2 \alpha_{j, j'}} + \frac{M^2 \sigma^4}{2 / \alpha_{j,
  j'}} \right) , \forall \alpha_{j, j'} > 0.
\end{align*}
We can choose $\alpha_{j, j'} > 0$ to make the term in the last step become
$\frac{n-1}{n}\sum_{k \geqslant j > j' \geqslant 0}^k \rho^{\frac{k - j'}{2}} M \sigma^2$ (by
applying inequality of arithmetic and geometric means).
Thus
\begin{align*}
  A_4 \leqslant &\frac{n-1}{n} \sum_{k \geqslant j > j' \geqslant 0}^k \rho^{\frac{k -
                  j'}{2}} M \sigma^2 \leqslant \frac{n-1}{n}M \sigma^2  \sum_{j' = 0}^k \sum_{j = j' + 1}^k
                  \rho^{\frac{k - j'}{2}}\\
  =&\frac{n-1}{n} M \sigma^2  \sum_{j' = 0}^k (k - j') \rho^{\frac{k
     - j'}{2}} \leqslant \frac{n-1}{n}M \sigma^2  \frac{\sqrt{\rho}}{\left( 1 - \sqrt{\rho}
  \right)^2} .
\end{align*}

Putting $A_3$ and $A_4$ back into $A_1$ we obtain:
\begin{align}
  A_1 \leqslant &\frac{n-1}{n} M\sigma^2 \left( \frac{1}{1 - \rho} + \frac{2
  \sqrt{\rho}}{\left( 1 - \sqrt{\rho} \right)^2} \right) = M\sigma^{2}\bar{\rho}. \label{eq:Wed Jul 19 11:15:27 EDT 2017}
\end{align}

We then start bounding $A_2$:
\begin{align*}
  A_2 = & \mathbb{E} \left\| \sum_{j = 0}^k \partial f (\hat{X}_{j}, i_j)
  \left( \frac{\mathbf{1}_n}{n} - \prod_{q = j + 1}^k W_q e_i \right)
  \right\|^2\\
  = & \mathbb{E}  \sum_{j = 0}^k \left\| \partial f (\hat{X}_{j}, i_j)
  \left( \frac{\mathbf{1}_n}{n} - \prod_{q = j + 1}^k W_q e_i \right)
  \right\|^2\\
  & + 2 \mathbb{E}  \sum_{j = 0}^k \sum_{j' = j + 1}^k \left\langle \partial
  f (\hat{X}_{j}, i_j) \left( \frac{\mathbf{1}_n}{n} - \prod_{q = j + 1}^k
  W_q e_i \right), \partial f (\hat{X}_{j'}, i_{j'}) \left(
  \frac{\mathbf{1}_n}{n} - \prod_{q = j' + 1}^k W_q e_i \right)
  \right\rangle\\
  \overset{\text{{\Cref{lemma:zhzlxdfmlqwr}},(\ref{eq:4zljdslaf})}}{\leqslant}
  &\frac{n-1}{n} \mathbb{E}  \sum_{j = 0}^k \left( \sum_{i = 1}^n p_i \| \nabla f_i (\hat{x}_{j}^{i})\|^2 \right) \rho^{k - j}\\
  & + 2 \mathbb{E}  \sum_{j = 0}^k \sum_{j' = j + 1}^k \| \partial f (\hat{X}_{j}, i_j)\|  \left\| \frac{\mathbf{1}_n}{n} - \prod_{q = j + 1}^k W_q
  e_i \right\|  \| \partial f (\hat{X}_{j'}, i_{j'})\|  \left\|
  \frac{\mathbf{1}_n}{n} - \prod_{q = j' + 1}^k W_q e_i \right\| . \numberthis
  \label{eq:Wed Jul 19 10:57:29 EDT 2017}
\end{align*}
For the second term:
\begin{align*}
  & \mathbb{E}  \sum_{j = 0}^k \sum_{j' = j + 1}^k \| \partial f (\hat{X}_{j}, i_j)\|  \left\| \frac{\mathbf{1}_n}{n} - \prod_{q = j + 1}^k W_q
  e_i \right\|  \| \partial f (\hat{X}_{j'}, i_{j'})\|  \left\|
  \frac{\mathbf{1}_n}{n} - \prod_{q = j' + 1}^k W_q e_i \right\|\\
  \leqslant & \mathbb{E}  \sum_{j = 0}^k \sum_{j' = j + 1}^k \left( \frac{\|
  \partial f (\hat{X}_{j}, i_j)\|^2 \| \partial f (\hat{X}_{j'},
  i_{j'})\|^2}{2 \alpha_{j, j'}} + \frac{\left\| \frac{\mathbf{1}_n}{n} -
  \prod_{q = j + 1}^k W_q e_i \right\|^2  \left\| \frac{\mathbf{1}_n}{n} -
  \prod_{q = j' + 1}^k W_q e_i \right\|^2}{2 / \alpha_{j, j'}} \right),
  \forall \alpha_{j, j'} > 0\\
  \overset{\text{{\Cref{lemma:zhzlxdfmlqwr}}}}{\leqslant} & \frac{1}{2}
  \mathbb{E}  \sum_{j \neq j'}^k \left( \frac{\| \partial f (\hat{X}_{j},
  i_j)\|^2 \| \partial f (\hat{X}_{j'}, i_{j'})\|^2}{2 \alpha_{j, j'}} +
  \frac{\rho^{k - \min \{j, j' \}}}{2 / \alpha_{j, j'}} \left(\frac{n-1}{n}\right)^{2} \right), \quad \forall
  \alpha_{j, j'} > 0, \alpha_{j, j'} = \alpha_{j', j} .
\end{align*}
By applying inequality of arithmetic and geometric means to the term in the
last step we can choose $\alpha_{j, j'} > 0$ such that
\begin{align*}
  & \mathbb{E}  \sum_{j = 0}^k \sum_{j' = j + 1}^k \| \partial f (\hat{X}_{j}, i_j)\|  \left\| \frac{\mathbf{1}_n}{n} - \prod_{q = j + 1}^k W_q
  e_i \right\|  \| \partial f (\hat{X}_{j'}, i_{j'})\|  \left\|
  \frac{\mathbf{1}_n}{n} - \prod_{q = j' + 1}^k W_q e_i \right\|\\
  \leqslant & \frac{n-1}{2n}  \mathbb{E}  \sum_{j \neq j'}^k \left( \| \partial f
  (\hat{X}_{j}, i_j)\|\| \partial f (\hat{X}_{j'}, i_{j'})\|
  \rho^{\frac{k - \min \{j, j' \}}{2}} \right)\\
  \leqslant &\frac{n-1}{2n} \mathbb{E}  \sum_{j \neq j'}^k \left( \frac{\|
  \partial f (\hat{X}_{j}, i_j)\|^2 +\| \partial f (\hat{X}_{j'},
  i_{j'})\|^2}{2} \rho^{\frac{k - \min \{j, j' \}}{2}} \right)\\
  = & \frac{n-1}{2n}  \mathbb{E}  \sum_{j \neq j'}^k \left( \| \partial f (\hat{X}_{j}, i_j)\|^2 \rho^{\frac{k - \min \{j, j' \}}{2}} \right) = \frac{n-1}{n}\sum_{j =
  0}^k \sum_{j' = j + 1}^k \left( \mathbb{E} \| \partial f (\hat{X}_{j},
  i_j)\|^2 \rho^{\frac{k - j}{2}} \right)\\
  = &\frac{n-1}{n} \sum_{j = 0}^k \left( \sum_{i = 1}^n p_i  \mathbb{E} \| \nabla f_i (\hat{x}_{j}^{i})\|^2 \right)  (k - j) \rho^{\frac{k - j}{2}} \numberthis
  \label{eq:Wed Jul 19 10:56:39 EDT 2017} .
\end{align*}
It follows from (\ref{eq:Wed Jul 19 10:56:39 EDT 2017}) and (\ref{eq:Wed Jul
19 10:57:29 EDT 2017}) that
\begin{align*}
  A_2 \leqslant & \frac{n-1}{n} \mathbb{E}  \sum_{j = 0}^k \left( \sum_{i = 1}^n p_i \|
  \nabla f_i (\hat{x}_{j}^{i})\|^2 \right)  \left( \rho^{k - j} + 2 (k - j)
  \rho^{\frac{k - j}{2}} \right)\\
  \overset{\text{{\Cref{lemma:20170614-001358}}}}{\leqslant} & \frac{n-1}{n}  \sum_{j = 0}^k
  \left( 12 L^2  \mathbb{E} \hat{M}_{j} + 6 \varsigma^2 + 2 \mathbb{E}
  \left\| \sum_{j' = 1}^n p_{j'} \nabla f_{j'} (\hat{x}_{j}^{j'}) \right\|^2
  \right)  \left( \rho^{k - j} + 2 (k - j) \rho^{\frac{k - j}{2}} \right)\\
  \leqslant & \frac{n-1}{n}\sum_{j = 0}^k \left( 12 L^2  \mathbb{E} \hat{M}_{j} + 2
  \mathbb{E} \left\| \sum_{j' = 1}^n p_{j'} \nabla f_{j'} (\hat{x}_{j}^{j'})
  \right\|^2 \right)  \left( \rho^{k - j} + 2 (k - j) \rho^{\frac{k - j}{2}} \right)\\
  & + 6 \varsigma^2  \underbrace{\frac{n-1}{n}\left( \frac{1}{1 - \rho} + \frac{2 \sqrt{\rho}}{\left(
  1 - \sqrt{\rho} \right)^2} \right)}_{=\bar{\rho}} .\numberthis \label{eq:Wed Jul 19 11:15:55 EDT 2017}
\end{align*}

Finally from \eqref{eq:Wed Jul 19 11:15:27 EDT 2017}, \eqref{eq:Wed Jul 19
  11:15:55 EDT 2017} and \eqref{eq:Wed Jul 19 11:16:27 EDT 2017} we obtain
\begin{align*}
  & \mathbb{E}  \left\| \frac{X_{k + 1}  \mathbf{1}_n}{n} - X_{k + 1} e_i
  \right\|^2\\
  \leqslant & 2 \gamma^2 A_1 + 2 M^{2}\gamma^2 A_2\\
  \leqslant & 2 \gamma^2 M \sigma^2  \bar{\rho}\\
  & + 2 \gamma^2 M^{2} \mathbb{E}  \sum_{j = 0}^k \left( 12 L^2 \hat{M}_{j} + 2
  \mathbb{E} \left\| \sum_{i = 1}^n p_i \nabla f_i (\hat{x}_{j}^{i})
  \right\|^2 \right)  \left( \rho^{k - j} + 2 (k - j) \rho^{\frac{k - j}{2}}
  \right) + 12 \gamma^2 M^{2}\varsigma^2  \bar{\rho}\\
  = & 2 \gamma^2  (M \sigma^2 + 6 M^{2} \varsigma^2)  \bar{\rho}\\
  & + 2\frac{n-1}{n} M^{2} \gamma^2  \mathbb{E}  \sum_{j = 0}^k \left( 12 L^2 \hat{M}_{j} + 2
  \mathbb{E} \left\| \sum_{i = 1}^n p_i \nabla f_i (\hat{x}_{j}^{i})
  \right\|^2 \right)  \left( \rho^{k - j} + 2 (k - j) \rho^{\frac{k - j}{2}}
  \right) .
\end{align*}

This completes the proof.

\end{proof}

\begin{lemma}
  \label{lemma:20170614-114822}
While $C_1 > 0$, we have
\begin{align*}
  \frac{\sum_{k = 0}^{K - 1} \mathbb{E} \hat{M}_{k}}{K} \leqslant & \frac{2
  \gamma^2  (M \sigma^2 + 6 M^2 \varsigma^2)  \bar{\rho} + \frac{4 \gamma^2
  M^2}{K}  \left( T \frac{n - 1}{n} + \bar{\rho} \right)  \sum_{k = 0}^{K - 1}
  \mathbb{E} \left\| \sum_{i = 1}^n p_i \nabla f_i (\hat{x}_{k}^{i})
  \right\|^2}{C_1}, \quad \forall K \ge 1.
\end{align*}
\end{lemma}

\begin{proof}
From \Cref{lemma:20170614-010128} and noting that $\hat{X}_{k} = X_{k-\tau_{k}}$, we have
\begin{align*}
  & \mathbb{E}  \left\| \frac{\hat{X}_{k}  \mathbf{1}_n}{n} - \hat{X}_{k} e_i \right\|^2\\
  \leqslant & 2 \gamma^2  (M \sigma^2 + 6 M^2 \varsigma^2)  \bar{\rho}\\
  & + 2 \frac{n - 1}{n} M^2 \gamma^2  \sum_{j = 0}^{k - \tau_k - 1} \left( 12
  L^2  \mathbb{E} \hat{M}_{j} + 2 \mathbb{E} \left\| \sum_{i = 1}^n p_i
  \nabla f_i (\hat{x}_{j}^{i}) \right\|^2 \right)  \left( \rho^{k - \tau_k -
  1 - j} + 2 (k - \tau_k - 1 - j) \rho^{\frac{k - \tau_k - 1 - j}{2}} \right)
  .
\end{align*}
By averaging from $i = 1$ to $n$ with distribution $\mathcal{I}$ we obtain
\begin{align*}
  & \mathbb{E} \hat{M}_{k} = \sum_{i = 1}^n p_i  \mathbb{E}  \left\|
  \frac{\hat{X}_{k }  \mathbf{1}_n}{n} - \hat{X}_{k} e_i \right\|^2\\
  \leqslant & 2 \gamma^2  (M \sigma^2 + 6 M^2 \varsigma^2)  \bar{\rho}\\
  & + 2 M^2 \gamma^2  \frac{n - 1}{n} \sum_{j = 0}^{k - \tau_k - 1} \left( 12
  L^2  \mathbb{E} \hat{M}_{j} + 2 \mathbb{E} \left\| \sum_{i = 1}^n p_i
  \nabla f_i (\hat{x}_{j}^{i}) \right\|^2 \right)  \left( \rho^{k - \tau_k -
  1 - j} + 2 (k - \tau_k - 1 - j) \rho^{\frac{k - \tau_k - 1 - j}{2}} \right)
  .
\end{align*}
It follows that
\begin{align*}
  \frac{\sum_{k = 0}^{K - 1} \mathbb{E} \hat{M}_{k}}{K} \leqslant & 2
  \gamma^2  (M \sigma^2 + 6 M^2 \varsigma^2)  \bar{\rho}\\
  & + \frac{2 \gamma^2}{K}  \frac{n - 1}{n} M^2  \sum_{k = 0}^{K - 1} \sum_{j
  = 0}^{k - \tau_k - 1} \begin{array}{c}
    \left( 12 L^2  \mathbb{E} \hat{M}_{j} + 2 \mathbb{E} \left\| \sum_{i =
    1}^n p_i \nabla f_i (\hat{x}_{j}^{i}) \right\|^2 \right) \times\\
    \left( \rho^{k - \tau_k - 1 - j} + 2 (k - \tau_k - 1 - j) \rho^{\frac{k -
    \tau_k - 1 - j}{2}} \right)
  \end{array}\\
  = & 2 \gamma^2  (M \sigma^2 + 6 M^2 \varsigma^2)  \bar{\rho}\\
  & + \frac{2 \gamma^2}{K}  \frac{n - 1}{n} M^2  \sum_{k = 0}^{K - 1} \sum_{j
  = 0}^{k - \tau_k - 1} \begin{array}{c}
    \left( 12 L^2  \mathbb{E} \hat{M}_{j} + 2 \mathbb{E} \left\| \sum_{i =
    1}^n p_i \nabla f_i (\hat{x}_{j}^{i}) \right\|^2 \right) \times\\
    \left( \rho^{\max \{ k - \tau_k - 1 - j, 0 \}} + 2 (\max \{ k - \tau_k - 1
    - j, 0 \}) \rho^{\frac{\max \{ k - \tau_k - 1 - j, 0 \}}{2}} \right)
  \end{array}\\
  \leqslant & 2 \gamma^2  (M \sigma^2 + 6 M^2 \varsigma^2)  \bar{\rho}\\
  & + \frac{2 \gamma^2}{K}  \frac{n - 1}{n} M^2  \sum_{j = 0}^{K - 1}
  \begin{array}{c}
    \left( 12 L^2  \mathbb{E} \hat{M}_{j} + 2 \mathbb{E} \left\| \sum_{i =
    1}^n p_i \nabla f_i (\hat{x}_{j}^{i}) \right\|^2 \right) \times\\
    \sum_{k = j + 1}^{\infty} \left( \rho^{\max \{ k - \tau_k - 1 - j, 0 \}} +
    2 (\max \{ k - \tau_k - 1 - j, 0 \}) \rho^{\frac{\max \{ k - \tau_k - 1 -
    j, 0 \}}{2}} \right)
  \end{array}\\
  \leqslant & 2 \gamma^2  (M \sigma^2 + 6 M^2 \varsigma^2)  \bar{\rho}\\
  & + \frac{2 \gamma^2}{K}  \frac{n - 1}{n} M^2  \sum_{j = 0}^{K - 1} \left(
  12 L^2  \mathbb{E} \hat{M}_{j} + 2 \mathbb{E} \left\| \sum_{i = 1}^n p_i
  \nabla f_i (\hat{x}_{j}^{i}) \right\|^2 \right)  \left( T + \sum_{h =
  0}^{\infty} \left( \rho^h + 2 h \rho^{\frac{h}{2}} \right) \right)\\
  \leqslant & 2 \gamma^2  (M \sigma^2 + 6 M^2 \varsigma^2)  \bar{\rho}\\
  & + \frac{2 \gamma^2}{K} M^2  \left( T \frac{n - 1}{n} + \bar{\rho} \right)
  \sum_{j = 0}^{K - 1} \left( 12 L^2  \mathbb{E} \hat{M}_{j} + 2 \mathbb{E}
  \left\| \sum_{i = 1}^n p_i \nabla f_i (\hat{x}_{j , i}) \right\|^2
  \right)\\
  \leqslant & 2 \gamma^2  (M \sigma^2 + 6 M^2 \varsigma^2)  \bar{\rho}\\
  & + \frac{4 \gamma^2 M^2 }{K} \left( T \frac{n - 1}{n} + \bar{\rho} \right)
  \sum_{k = 0}^{K - 1} \mathbb{E} \left\| \sum_{i = 1}^n p_i \nabla f_i (\hat{x}_{k}^{i}) \right\|^2\\
  & + \frac{24 L^2 \gamma^2 M^2 }{K}  \left( T \frac{n - 1}{n} + \bar{\rho}
  \right)  \sum_{k = 0}^{K - 1} \mathbb{E} \hat{M}_{k} .
\end{align*}
By rearranging the terms we obtain
\begin{align*}
  & \underbrace{\left( 1 - 24 L^2 M^2 \gamma^2  \left( T \frac{n - 1}{n} +
  \bar{\rho} \right) \right)}_{C_1} \frac{\sum_{k = 0}^{K - 1} \mathbb{E} \hat{M}_{k}}{K}\\
  \leqslant & 2 \gamma^2  (M \sigma^2 + 6 M^2 \varsigma^2)  \bar{\rho} +
  \frac{4 \gamma^2 M^2 }{K}  \left( T \frac{n - 1}{n} + \bar{\rho} \right)
  \sum_{k = 0}^{K - 1} \mathbb{E} \left\| \sum_{i = 1}^n p_i \nabla f_i (\hat{x}_{k}^{i}) \right\|^2,
\end{align*}
we complete the proof.
\end{proof}

\begin{lemma}
\label{lemma:20170614-115656}
For all $k\ge 0$ we have
\[ \mathbb{E} \left\| \frac{X_k  \mathbf{1}_n - \hat{X}_{k}
   \mathbf{1}_n}{n} \right\|^2 \leqslant \frac{\tau_k^2 \gamma^2 \sigma^2
   M}{n^2} + \tau_k \gamma^2  \sum_{t = 1}^{\tau_k} \left( \frac{M^2}{n^2}
   \sum_{i = 1}^n p_i  \mathbb{E} \| \nabla f_i (\hat{x}_{k - t}^{i})\|^2 \right) . \]
\end{lemma}
\begin{proof}
  \begin{align*}
  \mathbb{E} \left\| \frac{X_k  \mathbf{1}_n - \hat{X}_{k}
  \mathbf{1}_n}{n} \right\|^2 \overset{\text{\Cref{ass:20170614-002451}-7}}{=} & \mathbb{E} \left\| \frac{\sum_{t=1 }^{\tau_{k}} \gamma \partial g (\hat{X}_{k - t} ; \xi_{k - t}^{i_{k-t}}, i_{k - t}) \mathbf{1}_n}{n}
  \right\|^2\\
  \leqslant & \tau_k  \sum_{t = 1}^{\tau_k} \gamma^2  \mathbb{E} \left\|
  \frac{\partial g (\hat{X}_{k - t} ; \xi_{k - t}^{i_{k-t}}, i_{k
  - t}) \mathbf{1}_n}{n} \right\|^2\\
  \overset{\text{{\Cref{lemma:zhlkdsamfls}}}}{\leqslant} & \tau_k  \sum_{t =
  1}^{\tau_k} \gamma^2  \left( \frac{\sigma^2 M}{n^2} + \frac{M^2}{n^2}
  \sum_{i = 1}^n p_i  \mathbb{E} \left\| \nabla f_i (\hat{x}_{k - t}^{i})\right\|^2 \right) ,
\end{align*}
where the first step comes from any $n\times n$ doubly stochastic matrix multiplied by
$\mathbf{1}_{n}$ equals $\mathbf{1}_{n}$ and \Cref{ass:20170614-002451}-7.
\end{proof}

\begin{proof}[Proof to \Cref{coro:Sun Aug 13 00:14:48 EDT 2017}]
  To prove this result, we will apply \Cref{thm:main-thm}. We first verify that
  all conditions can be satisfied in \Cref{thm:main-thm}.

First $C_1 > 0$ can be satisfied by a stronger condition $C_1 \geq 1 / 2$
which can be satisfied by $\gamma \leqslant \frac{1}{4 \sqrt{6} ML}  \left( T
\frac{n - 1}{n} + \bar{\rho} \right)^{- 1 / 2}$. Second $C_3 \leq 1$ can be
satisfied by : \[\gamma \leqslant \min \left\{ \frac{n}{8 MT^2 L}, \frac{1}{8
\sqrt{3} LM}  \bar{\rho}^{- 1 / 2}, \frac{1}{32 nML}  \bar{\rho}^{- 1},
\frac{n^{1 / 3}}{8 \sqrt{6} MLT^{2 / 3}}  \bar{\rho}^{- 1 / 3} \right\}\] and
$C_1 \ge 1 / 2$, which can be seen from
\begin{align*}
  C_3 = & \frac{1}{2} + \frac{2 \left( 6 \gamma^2 L^2 M^2 + \gamma nML +
  \frac{12 M^3 L^3 T^2 \gamma^3}{n} \right)  \bar{\rho}}{C_1} + \frac{LT^2
  \gamma M}{n}\\
  \overset{C_1 \geqslant \frac{1}{2}}{\leqslant} & \frac{1}{2} + 24 \gamma^2
  L^2 M^2 + 4 \gamma nML + \frac{48 M^3 L^3 T^2 \gamma^3}{n}  \bar{\rho} +
  \frac{LT^2 \gamma M}{n} .
\end{align*}
The requirements on $\gamma$ are given by making each of the last four terms
smaller than $1 / 8$:
\begin{align*}
  \frac{LT^2 \gamma M}{n} \leqslant & \frac{1}{8} \Longleftarrow \gamma
  \leqslant \frac{n}{8 MT^2 L},\\
  24 \gamma^2 L^2 M^2  \bar{\rho} \leqslant & \frac{1}{8} \Longleftarrow
  \gamma \leqslant \frac{1}{8 \sqrt{3} LM}  \bar{\rho}^{- 1 / 2},\\
  4 \gamma nML \bar{\rho} \leqslant & \frac{1}{8} \Longleftarrow \gamma
  \leqslant \frac{1}{32 nML}  \bar{\rho}^{- 1},
\end{align*}
and
\begin{align*}
  \frac{48 M^3 L^3 T^2 \gamma^3}{n}  \bar{\rho} \leqslant & \frac{1}{8}
  \Longleftarrow \gamma \leqslant \frac{n^{1 / 3}}{8 \sqrt{6} MLT^{2 / 3}}
  \bar{\rho}^{- 1 / 3} .
\end{align*}
Third $C_2 \geqslant 0$ can be satisfied by \[\gamma \le \min \left\{
\frac{n}{10 LM}, \frac{n}{2 \sqrt{5} MLT}, \frac{n^{1 / 3}}{8 ML}  \left( T
\frac{n - 1}{n} + \bar{\rho} \right)^{- 1 / 3}, \frac{1}{4 \sqrt{5} LM}
\left( T \frac{n - 1}{n} + \bar{\rho} \right)^{- 1 / 2}, \frac{n^{1 / 2}}{6
MLT^{1 / 2}}  \left( T \frac{n - 1}{n} + \bar{\rho} \right)^{- 1 / 4}
\right\}\] and $C_1 \ge 1 / 2$, which can be seen from
\begin{align*}
  C_2 := & \frac{\gamma M}{2 n} - \frac{\gamma^2 LM^2}{n^2} - \frac{2 M^3 L^2
  T^2 \gamma^3}{n^3} - \left( \frac{6 \gamma^2 L^3 M^2}{n^2} + \frac{\gamma
  M}{n} L^2 + \frac{12 M^3 L^4 T^2 \gamma^3}{n^3} \right)  \frac{4 M^2
  \gamma^2  (T \frac{n - 1}{n} + \bar{\rho})}{C_1} \geqslant 0\\
  \overset{C_1 \geqslant \frac{1}{2}}{\Longleftarrow} & 1 \geqslant \frac{2
  \gamma LM}{n} + \frac{4 M^2 L^2 T^2 \gamma^2}{n^2} + \frac{96 \gamma L^3
  M}{n} + 16 L^2 + \frac{192 M^2 L^4 T^2 \gamma^2}{n^2} M^2 \gamma^2  (T
  \frac{n - 1}{n} + \bar{\rho}) .
\end{align*}
The last inequality is satisfied given the requirements on $\gamma$ because
each term on the RHS is bounded by $1 / 5$:
\begin{align*}
  \frac{2 \gamma LM}{n} \leqslant & \frac{1}{5} \Longleftarrow \gamma
  \leqslant \frac{n}{10 LM},\\
  \frac{4 M^2 L^2 T^2 \gamma^2}{n^2} \leqslant & \frac{1}{5} \Longleftarrow
  \gamma \leqslant \frac{n}{2 \sqrt{5} MLT},\\
  \frac{96 \gamma L^3 M}{n} M^2 \gamma^2  (T \frac{n - 1}{n} + \bar{\rho})
  \leqslant & \frac{1}{5} \Longleftarrow \gamma \leqslant \frac{n^{1 / 3}}{8
  ML}  (T \frac{n - 1}{n} + \bar{\rho})^{- 1 / 3},\\
  16 L^2 M^2 \gamma^2  (T \frac{n - 1}{n} + \bar{\rho}) \leqslant &
  \frac{1}{5} \Longleftarrow \gamma \leqslant \frac{1}{4 \sqrt{5} LM}  (T
  \frac{n - 1}{n} + \bar{\rho})^{- 1 / 2},\\
  \frac{192 M^2 L^4 T^2 \gamma^2}{n^2} M^2 \gamma^2  (T \frac{n - 1}{n} +
  \bar{\rho}) \leqslant & \frac{1}{5} \Longleftarrow \gamma \leqslant
  \frac{n^{1 / 2}}{6 MLT^{1 / 2}}  (T \frac{n - 1}{n} + \bar{\rho})^{- 1 / 4}
  .
\end{align*}
Combining all above the requirements on $\gamma$ to satisfy $C_1 \ge 1 / 2,
C_2 \ge 0$ and $C_3 \le 1$ are
\[ \gamma \leqslant \frac{1}{ML} \min \left\{ \begin{array}{c}
     \frac{1}{4 \sqrt{6}}  \left( T \frac{n - 1}{n} + \bar{\rho} \right)^{- 1
     / 2}, \frac{n}{8 T^2}, \frac{1}{8 \sqrt{3}}  \bar{\rho}^{- 1 / 2},\\
     \frac{1}{32 n}  \bar{\rho}^{- 1}, \frac{n^{1 / 3}}{8 \sqrt{6} T^{2 / 3}}
     \bar{\rho}^{- 1 / 3},\\
     \frac{n}{10}, \frac{n}{2 \sqrt{5} T}, \frac{n^{1 / 3}}{8}  \left( T
     \frac{n - 1}{n} + \bar{\rho} \right)^{- 1 / 3},\\
     \frac{1}{4 \sqrt{5}}  \left( T \frac{n - 1}{n} + \bar{\rho} \right)^{- 1
     / 2}, \frac{n^{1 / 2}}{6 T^{1 / 2}}  \left( T \frac{n - 1}{n} +
     \bar{\rho} \right)^{- 1 / 4}
   \end{array} \right\} . \]
Note that the RHS is larger than
\begin{align*}
  U := & \frac{1}{ML} \min \left\{ \begin{array}{c}
    \frac{1}{8 \sqrt{3}  \sqrt{T \frac{n - 1}{n} + \bar{\rho}}}, \frac{n}{8
    T^2}, \frac{1}{32 n \bar{\rho}}, \frac{n}{10}, \frac{n^{1 / 12} (n - 1)^{-
    1 / 4}}{\left( 8 \sqrt{6} T^{2 / 3} + 8 \right)  \left( T + \bar{\rho}
    \frac{n}{n - 1} \right)^{1 / 3}}
  \end{array} \right\} .
\end{align*}
Let $\gamma = \frac{n}{10 ML + \sqrt{\sigma^2 + 6 M \varsigma^2}  \sqrt{KM}}$
then if $\gamma \leqslant U$ we will have $C_1 \ge 1 / 2, C_2 \ge 0$ and $C_3
\le 1$. Further investigation gives us
\begin{align*}
  \gamma = \frac{n}{10 ML + \sqrt{\sigma^2 + 6 M \varsigma^2}  \sqrt{KM}}
  \leqslant & \frac{1}{ML} \min \left\{ \begin{array}{c}
    \frac{1}{8 \sqrt{3}  \sqrt{T \frac{n - 1}{n} + \bar{\rho}}}, \frac{n}{8
    T^2}, \frac{1}{32 n \bar{\rho}},\\
    \frac{n}{10}, \frac{n^{1 / 12} (n - 1)^{- 1 / 4}}{\left( 8 \sqrt{6} T^{2 /
    3} + 8 \right)  \left( T + \bar{\rho} \frac{n}{n - 1} \right)^{1 / 3}}
  \end{array} \right\}\\
  \Longleftarrow 10 ML + \sqrt{\sigma^2 + 6 M \varsigma^2}  \sqrt{KM}
  \geqslant & nML \max \left\{ \begin{array}{c}
    8 \sqrt{3}  \sqrt{T \frac{n - 1}{n} + \bar{\rho}}, \frac{8 T^2}{n}, 32 n
    \bar{\rho},\\
    \frac{\left( 8 \sqrt{6} T^{2 / 3} + 8 \right)  \left( T + \bar{\rho}
    \frac{n}{n - 1} \right)^{1 / 3}}{n^{1 / 12} (n - 1)^{- 1 / 4}}
  \end{array} \right\}\\
  \Longleftarrow K \geqslant & \frac{ML^2 n^2}{\sigma^2 + 6 M \varsigma^2}
  \max \left\{ \begin{array}{c}
    192 \left( T \frac{n - 1}{n} + \bar{\rho} \right), \frac{64 T^4}{n^2},
    1024 n^2  \bar{\rho}^2,\\
    \frac{\left( 8 \sqrt{6} T^{2 / 3} + 8 \right)^2  \left( T + \bar{\rho}
    \frac{n}{n - 1} \right)^{2 / 3}}{n^{1 / 6} (n - 1)^{- 1 / 2}}
  \end{array} \right\} .
\end{align*}
It follows from {\Cref{thm:main-thm}} that if the last inequality is
sastisfied and $\gamma = \frac{n}{10 ML + \sqrt{\sigma^2 + 6 M \varsigma^2}
\sqrt{KM}}$, we have
\begin{align*}
  \frac{\sum_{k = 0}^{K - 1} \mathbb{E} \left\| \nabla f \left( \frac{X_k
  \mathbf{1}_n}{n} \right) \right\|^2}{K} \leqslant & \frac{2 (\mathbb{E} f
  (x_0) - f^{\ast}) n}{\gamma KM} + \frac{2 \gamma L}{Mn}  (M \sigma^2 + 6 M^2
  \varsigma^2)\\
  = & \frac{20 (\mathbb{E} f (x_0) - f^{\ast}) L}{K} + \frac{2 (\mathbb{E} f
  (x_0) - f^{\ast})  \sqrt{\sigma^2 + 6 M \varsigma^2}}{\sqrt{KM}}\\
  & + \frac{2 L}{M \left( 10 ML + \sqrt{\sigma^2 + 6 M \varsigma^2}
  \sqrt{KM} \right)}  (M \sigma^2 + 6 M^2 \varsigma^2)\\
  \leqslant & \frac{20 (\mathbb{E} f (x_0) - f^{\ast}) L}{K} + \frac{2
  (\mathbb{E} f (x_0) - f^{\ast} + L)  \sqrt{\sigma^2 + 6 M
  \varsigma^2}}{\sqrt{KM}} .
\end{align*}
This completes the proof.
\end{proof}

\end{document}